\documentclass[11pt, reqno]{amsart}
\usepackage{mathrsfs}
\usepackage[active]{srcltx}
\usepackage{mathrsfs}
\usepackage{mathtools}
\usepackage{longtable}

\usepackage{xcolor}
\definecolor{bluecite}{HTML}{0875b7}
\usepackage[unicode=true,
bookmarksopen={true},
pdffitwindow=true,
colorlinks=true,
linkcolor=bluecite,
citecolor=bluecite,
urlcolor=bluecite,
hyperfootnotes=false,
pdfstartview={FitH},
pdfpagemode= UseNone]{hyperref}

\newtheorem{example}{Example}[section]
\newtheorem{proposition}{Proposition}[section]
\newtheorem{theorem}{Theorem}[section]
\newtheorem{lemma}{Lemma}[section]
\newtheorem{corollary}{Corollary}[section]

\newtheorem{remark}{Remark}[section]
\numberwithin{equation}{section}

\usepackage{bbm}  
\usepackage{dsfont}

\address{\textsc{Zolt\'an Balogh}: 
	Mathematisches Institut, University of Bern, Sidlerstrasse 12, 3012, Bern, Switzerland.}
	\email{zoltan.balogh@unibe.ch}
\address{\textsc{Sebastiano Don}:
Mathematisches Institut, University of Bern, Sidlerstrasse 12, 3012, Bern, Switzerland.}
\email{sebastiano.don@unibe.ch}
\address{\textsc{Alexandru Krist\'aly}: 
Department of Economics, Babe\c s-Bolyai University, Cluj-Napoca, Romania  \& Institute of Applied Mathematics, \'Obuda
	University, 
	Budapest, Hungary 
}
\email{alex.kristaly@econ.ubbcluj.ro; kristaly.alexandru@uni-obuda.hu
}

\thanks{Z. M. Balogh and S. Don were
	supported by the Swiss National Science Foundation, Grant Nr. {200020\_191978}.  
	A. Krist\'aly  was supported by the Excellence Researcher Program \'OE-KP-2-2022 of  \'Obuda University, Hungary, and by the 
	UEFISCDI/CNCS grant PN-III-P4-ID-PCE2020-1001, Romania.
}

\title[Sharp weighted GN inequalities]{Weighted Gagliardo-Nirenberg inequalities via Optimal Transport Theory and Applications}

\date{\today}

\author[Z.\ M.\ Balogh, S.\ Don and A.\ Krist\'aly]{Zolt\'an M. Balogh, Sebastiano Don and Alexandru Krist\'aly}

\begin{document}
	\begin{abstract}
		We prove Gagliardo-Nirenberg inequalities with three  weights -- verifying a joint concavity condition --   on open convex cones of $\mathbb R^n$. If the weights are equal to each other the inequalities become sharp and we compute explicitly the sharp constants. For a certain range of parameters we can characterize the class of extremal functions; in this case, we also show that the sharpness in the main three-weighted Gagliardo-Nirenberg inequality implies that the weights must be equal up to some constant multiplicative factors. 
		Our approach uses optimal mass transport theory and a careful analysis of the joint concavity condition of the weights. As applications we establish sharp weighted $p$-log-Sobolev, Faber-Krahn and isoperimetric inequalities with explicit sharp constants. 
	\end{abstract}
	\maketitle 
	\tableofcontents

\vspace{-1cm}
\section{Introduction}

Schwarz-type symmetrization techniques, that are instrumental e.g.\ in the P\'olya-Szeg\H o inequality, are the classical tools to prove Sobolev-type inequalities both in Euclidean and in curved spaces, see Talenti \cite{Talenti}, Aubin \cite{Aubin} and Hebey \cite{Hebey}. More recently,  further methods have been provided to prove such functional inequalities, even in more general settings involving weights and normed spaces. The first remarkable  result in this direction is attributed to Cordero-Erausquin, Nazaret and Villani \cite{CENV}, who proved sharp Sobolev and Gagliardo-Nirenberg inequalities in $\mathbb R^n$ via  optimal mass transport theory (for short, OMT). A nice exposition of this method can be found in the book of Villani \cite{Villani}. An alternative method is provided by Cabr\'e,  Ros-Oton and Serra \cite{Cabre-Ros-Oton, Cabre-Ros-Oton-Serra}, where sharp weighted Sobolev-type inequalities are established  by the ABP-method. In fact, in \cite{Cabre-Ros-Oton} the weights are considered to be monomials defined on certain convex open cones of $\mathbb R^n$, arising from the theory of reaction-diffusion problems in domains with symmetry of double revolution. 

More recently, by using OMT,  Lam  \cite{Lam} established the validity of the following one-weighted version of the  Gagliardo-Nirenberg inequality  \begin{equation}\label{eq:Lam1-0}
	\left(\int_E |u|^{\alpha p}\omega dx\right)^{\frac{1}{\alpha p}}\leq  C_1 \left(\int_E|\nabla u|^p\omega dx\right)^{\frac{\theta_1}{p}}\left(\int_E|u|^{\alpha p \gamma}\omega dx\right)^{\frac{1-\theta_1}{\alpha p \gamma}},\ \ u\in C_c^\infty(\mathbb R^n),
\end{equation} 
together with its dual, where $C_1>0$,  $1-\frac{1}{n+\tau}\leq \gamma < 1$,  $1<p<n+\tau$, $\alpha=\frac 1{p(\gamma-1)+1}$ and $\theta_1=\frac{(n+\tau)(1-\gamma)}{\alpha \gamma (p-\tau-n)+n+\tau}$. Here, $E\subseteq \mathbb R^n$ is an open convex cone and  $\omega\colon E\to (0,+\infty)$ is a homogeneous weight  with degree $\tau\geq 0$ verifying the condition 
\begin{equation}\label{Lam-1-0}
	\frac{1}{1-\gamma}\left(\frac{\omega(\nabla \varphi(x))}{\omega(x)}\right)^{1-\gamma}({\rm det}(M))^{1-\gamma}\leq \frac{1}{1-\gamma}-(\tau+n)+\frac{ \nabla \omega(x)\cdot\nabla \varphi(x)}{\omega (x)}+{\rm tr}(M),
\end{equation}
for  any positive definite symmetric matrix $M$ and locally Lipschitz function $\varphi$ with $\nabla \varphi(x)\in \overline E$ for any $x\in \overline E$ and $\nabla \varphi\cdot \textbf{n}\leq 0$ on $\partial E$. Hereafter, $ \textbf{n}(x)$ stands for the outer normal vector at $x\in \partial E$.  A closer analysis shows that particular cases of  Lam's Gagliardo-Nirenberg inequality  \eqref{eq:Lam1-0}    provide the same (sharp) results of Cordero-Erausquin, Nazaret and Villani \cite{CENV} in the unweighted case, and those of Cabr\'e,  Ros-Oton and Serra \cite{Cabre-Ros-Oton, Cabre-Ros-Oton-Serra} and Nguyen \cite{Nguyen} for monomial weights of the form $\omega(x)=x_1^{a_1}\dots x_n^{a_n}$ on $E=\{x=(x_1,\dots,x_n)\in \mathbb R^n: x_i>0\ {\rm whenever}\ a_i>0\}$. 

We shall see in section \S \ref{sec:Preliminar-0}, that the condition \eqref{Lam-1-0} can be formulated in a simpler form. It  is in fact equivalent to the  $\frac{1-\gamma}{1-n(1-\gamma)}$-concavity of $\omega$ on $E$,  and to the inequality:

\begin{equation}\label{Lam-equiv}
		\left(\frac{1}{1-\gamma}-n\right)\left(\frac{\omega(y)}{\omega(x)}\right)^\frac{1-\gamma}{1-n(1-\gamma)}\leq \frac{1}{1-\gamma}-(n+\tau)+\frac{ \nabla \omega(x)\cdot y}{\omega (x)}, \ x, y  \in E.
		\end{equation}

Due to this observation,  the Gagliardo-Nirenberg inequality \eqref{eq:Lam1-0} can be related to the weighted Sobolev inequality proved by  Ciraolo, Figalli and  Roncoroni \cite{CFR} for a homogeneous and  $1/\tau$-concave weight $\omega$ on $E$, where $\frac 1\tau =\frac{1-\gamma}{1-n(1-\gamma)}$. 

The goal of the present paper is to consider Gagliardo-Nirenberg-type inequalities with \textit{different weights}. Such inequalities are motivated by applications in mathematical physics and sub-Riemannian geometry. Indeed, in order to study fast diffusion problems, Bonforte, Dolbeault,  Muratori and Nazaret \cite{Bonforte-etal-1, Bonforte-etal-2, Dolbeault-etal} established Gagliardo-Nirenberg inequalities with different power-law weights. Furthermore, Sobolev-type inequalities with two different weights appear both in reaction-diffusion equations, see Cabr\'e,  Ros-Oton and Serra \cite{Cabre-Ros-Oton, Cabre-Ros-Oton-Serra} and Castro \cite{Castro}, and in the axially symmetric reduction of sub-Riemannian Sobolev-type inequalities, see  Balogh, Guti\'errez and Krist\'aly \cite{BGK_PLMS}. 

Recall that in \cite{BGK_PLMS} a \textit{two-weighted}  Sobolev-type inequality has been established. To formulate this result, denote by $E\subseteq \mathbb R^n$ an open convex cone and  consider two homogeneous weights $\omega_1,\omega_2\colon E \to (0,+\infty)$ with degrees $\tau_1$ and $\tau_2$, respectively. Let $q$ be defined by 
\[
\frac{\tau_1+n}{q}=\frac{\tau_2+n}{p}-1;
\]
Assume, that  $\tau_2>(1-\frac pn)\tau_1$, and define the fractional dimension $n_\tau\geq n$ as given by  $\frac 1n_\tau=\frac 1p-\frac 1q$. 
The main result of \cite{BGK_PLMS} states that if there exists a constant $C_{0}>0$ such that the joint concavity condition 
\begin{equation} \label{PLMS-cond}
\left(\left(\frac{\omega_2(y)}{\omega_2(x)}\right)^{\frac 1p}\left(\frac{\omega_1(x)}{\omega_1(y)}\right)^{\frac1q}\right)^{\frac{n_\tau}{n_\tau-n}}\leq C_0 \left(\frac1{p'} \frac{\nabla \omega_1(x)}{\omega_1(x)}+\frac{1}{p}\frac{\nabla \omega_2(x)}{\omega_2(x)}\right)\cdot y,\ \  \forall x,y\in E
\end{equation}
is satisfied,  then there exists a constant $C_{1}>0$ such that the  following two-weighted Sobolev inequality holds: 
\begin{equation}\label{eq:PLMS}
\left(\int_E |u|^{q}\omega_{1} dx\right)^{\frac{1}{q}}\leq C_1 \left(\int_E|\nabla u|^p\omega_{2} dx\right)^{\frac{1}{p}} ,  u\in C_c^\infty(\mathbb R^n).
\end{equation}

 The aim of the present paper is to provide a common extension of the two conditions \eqref{Lam-equiv} and \eqref{PLMS-cond} and their consequences \eqref{eq:Lam1-0} and \eqref{eq:PLMS}, respectively. More precisely we consider Gagliardo-Nirenberg inequalities for  \textit{three}   homogeneous  weights $\omega_i\colon E\to (0,+\infty)$ of class $\mathcal C^1$, $i\in \{1,2,3\},$ having their degree of homogeneity $\tau_i>-n$, (i.e.\ $\omega_{i}(tx) = t^{\tau_{i}}\omega_{i}(x)$ for $t>0, x \in E$) , where $E\subseteq \mathbb R^n$ is an open convex cone, $n\geq 2$. 
 To do this, we formulate the condition connecting these weights, which plays the central role in our study.\\ 

\noindent {\sc Main Condition}: \textit{Let  $\gamma> 1-\frac 1n$ with $\gamma\neq 1$ and $1<p < \infty$. The triplet $(\omega_1,\omega_2,\omega_3)$  satisfies  {\rm  condition (C) on $E$} if there exist $K\in \mathbb R$  and $C_0>0$ such that
\begin{equation}\label{eq:conditionC}
\begin{aligned}
	\left(\frac 1{1-\gamma}-n\right)\left(\left(\frac{\omega_2(y)}{\omega_2(x)}\right)^{\frac 1p}\left(\frac{\omega_1(y)}{\omega_1(x)}\right)^{\frac1{p'}-\gamma}\right)^{\frac1{1-n(1-\gamma)}}\leq& \left(\frac 1{1-\gamma}+K\right)\frac{\omega_3(x)}{\omega_1^{1/{p'}}(x)\omega_2^{1/p}(x)}\\
	&+C_0\left(\frac 1{p}\frac{\nabla\omega_2(x)}{\omega_2(x)}+\frac{1}{p'}\frac{\nabla\omega_1(x)}{\omega_1(x)}\right)\cdot y,
	\end{aligned}
\end{equation}
for every $x,y\in E$.}\\
 
 Hereafter, $p'=p/(p-1)$ is the conjugate of $p>1$, and ``$\cdot$'' stands for the usual inner product in $\mathbb R^n$. Various examples verifying condition (C)  will be discussed in Section \ref{sec:Preliminar-0}. 

 Notice that \eqref{eq:conditionC} is a natural extension of \eqref{Lam-equiv} and \eqref{PLMS-cond} when three weights are taken into account. More precisely, 
	\begin{itemize}
		\item if $\omega_1=\omega_2=\omega_3=:\omega$ in \eqref{eq:conditionC}, $K=-n-\tau$ and $C_0=1$, we recover \eqref{Lam-equiv}, which is equivalent to the fact that the triplet $(E,|\cdot|,\omega)$ verifies the curvature-dimension condition $CD(0,n+\tau)$;
		\item if $\gamma<1$ and $K=-\frac1{1-\gamma}$, from \eqref{eq:conditionC} we obtain \eqref{PLMS-cond} with the identities
		$
		\frac 1q=\gamma-\frac1{p'}$, $ \frac 1{n_\tau}=1-\gamma=\frac 1p-\frac 1q.
		$
	\end{itemize}

Let us begin with some notations. Let $p>1$,  $\alpha>0$ and $\gamma>0$ be such that 
$\alpha=\frac 1{p(\gamma-1)+1}.$ We denote by  $L^p(\omega;E)= \{ u: E \to \mathbb R:  \int_{E} |u(x)|^{p}\omega(x) dx < +\infty \} $ and introduce the weighted Sobolev space  
\begin{equation}\label{W-p-q-0}
	\dot{W}^{p,\alpha }(\omega_1,\omega_2;E):=\{u\in L^{\alpha p}(\omega_1;E):|\nabla u| \in L^p(\omega_2;E)\},
\end{equation}
and
\[
\mathcal G_{\omega_1,\omega_2}\coloneqq\left\{G\colon E\to [0,+\infty): G\in L^1(\omega_1;E)\cap L^\gamma(\omega_1^\frac{1}{p'}\omega_2^\frac{1}{p};E),\ \int_E G(y)|y|^{p'}\omega_1(y)dy<+\infty\right\}.
\]
To formulate our first main result we introduce the following constants
\begin{equation}\label{L-M}
	L:=-(n+\tau_1)\gamma+n+\tau_3, \ M:=p+\frac{n+\tau_1}{\alpha}-(n+\tau_2),
\end{equation}
\begin{equation}\label{C-1-1-constant-initial}
	C_{K,L,M,C_0}=\left(\frac{1}{1-\gamma}+K\right)^\frac{M}{p}C_0^{L-(1-\gamma)n\left(\frac{M}{p}+L\right)}(\gamma \alpha)^L(1-\gamma)^{\frac{M}{p}+L}\frac{(M+pL)^{\frac{M}{p}+L}}{M^\frac{M}{p}L^L},
\end{equation}
and 
\begin{equation}\label{C-1-2-constant-initial}
	C_{\mathcal G_{\omega_1,\omega_2}}=\inf_{G\in \mathcal G_{\omega_1,\omega_2};\int_E G\omega_1=1}\frac{\displaystyle\left(\int_E G(y)|y|^{p'}\omega_1(y)dy\right)^\frac{L}{p'}}{\displaystyle\left(\int_E G^\gamma(y)\omega_1^\frac{1}{p'}(y)\omega_2^\frac{1}{p}(y)dy\right)^{\frac{M}{p}+L}}.
\end{equation}

Our main Gagliardo-Nirenberg inequalities with three weights will be proved by the OMT method and are stated as follows; we first deal with the case $\gamma<1$:

%\begin{definition}\label{def:conditionC}
%	Let $\omega_1,\omega_2,\omega_3\colon E\to (0,\infty)$ be three differentiable functions,
%\end{definition}

\begin{theorem}\label{th:CKN2weights} $(\gamma<1)$ Let $n\geq 2$, $1<p< \infty$, $1>\gamma>\max\{1-1/n,1/{p'}\}$ and $E\subseteq \mathbb R^n$ be an open convex cone, and  $\omega_i\colon E\to (0,+\infty)$ be homogeneous weights  with degree $\tau_i>-n$ and of class $\mathcal C^1$, $i\in \{1,2,3\}$, such that the triplet  $(\omega_1,\omega_2,\omega_3)$ satisfies  condition \textbf{\rm (C)} on $E$. 
%	that satisfy \eqref{eq:conditionC} with parameters $p> 1$, $\gamma> \frac 1{p'}$, $\gamma\neq 1$,  $K \in \mathbb R$ and $C_0>0$.  
	Let $\alpha=\frac 1{p(\gamma-1)+1}$ and assume that $L>0$ and $M\geq 0$. 
	 Then if $ \frac{1}{1-\gamma}+K >0$ we have 
		\begin{equation}\label{eq:CKN2weights}
			\left(\int_E |u|^{\alpha p}\omega_1dx\right)^{\frac{1}{\alpha p}}\leq {\sf C}_1 \left(\int_E|\nabla u|^p\omega_2 dx\right)^{\frac{\theta_1}{p}}\left(\int_E|u|^{\alpha p \gamma}\omega_3 dx\right)^{\frac{1-\theta_1}{\alpha p \gamma}},\ \forall u\in \dot{W}^{p,\alpha}(\omega_1,\omega_2;E),
		\end{equation} 
	 where  
	 \begin{equation}\label{C-1-optimal}
	 	 {\sf C}_1=\left(C_{K,L,M,C_0}C_{\mathcal G_{\omega_1,\omega_2}}\right)^\frac{1}{\alpha\gamma M+L} \  and\  \theta_1=\frac{L}{\alpha\gamma M+L}\in (0,1],
	 \end{equation}
%	 \frac{-(n+\tau_1)\gamma+n+\tau_3}{\alpha \gamma (p-\tau_2-n)+n+\tau_3}, 
If $\frac{1}{1-\gamma}+K =0$ we have that $M=0$ and the above formula still holds by using the convention $0^0=1$.

\end{theorem}

\noindent We notice that when  $\displaystyle \int_E|u|^{\alpha p \gamma}\omega_3 dx=+\infty$ in \eqref{eq:CKN2weights}, we have nothing to prove in \eqref{eq:CKN2weights}; thus, we may consider without loss of generality that it is finite. Note, that in the limit case $K=-\frac{1}{1-\gamma}$ (and consequently,  $M=0$ and $\theta_1=1$),  $\omega_3$ disappears from condition (C). In this case, setting  with $\gamma=1-\frac{1}{n_\tau}$,  where $n_\tau$ is the fractional dimension, Theorem \ref{th:CKN2weights} reduces to the Sobolev inequality on $E$ with two weights, stated above (see   \cite[Theorem 1.1]{BGK_PLMS}). One can also observe that when $\omega_1=\omega_2=\omega_3=:\omega$,  $K=-n-\tau$ and $C_0=1$, where $\tau$ is the degree of homogeneity of the weight $\omega$, Theorem \ref{th:CKN2weights} reduces to the above mentioned main result of Lam \cite[Theorem 1.7]{Lam}, see also Theorem \ref{th:Lam} (and Proposition \ref{Lam-equivalence}). 

Since we obtain a sharp result  in the case when the three weights coincide (see Theorem \ref{th:Lam}), the  question of the converse implication naturally arises for the Gagliardo-Nirenberg inequalities: assuming sharpness and the existence of non-zero extremals in Theorem~\ref{th:CKN2weights} (see  inequality \eqref{eq:CKN2weights}), are the weights related to each other?
The answer turns out to be affirmative and it reads as follows.
\begin{theorem}\label{prop:equalitycase-intro}
	Under the same assumptions as in Theorem \ref{th:CKN2weights},  if  
	 equality holds 
	%either
	in \eqref{eq:CKN2weights} for some 
	%$u\in \dot{W}^{p,\alpha \gamma}(\omega_3,\omega_2;E)\setminus \{0\}$ or in  \eqref{eq:CKN2weightsgamma>1} for some 
	$u\in \dot{W}^{p,\alpha }(\omega_1,\omega_2;E)\setminus \{0\}$ and the infimum in \eqref{C-1-2-constant-initial} is achieved, then the following statements hold.  
	\begin{itemize}
		\item[(i)] 	If $\frac{1}{1-\gamma} + K \neq 0$, then the  weights $\omega_1,$  $\omega_2$  and $\omega_3$  are equal up to some constant multiplicative factors and $\frac{1-\gamma}{1-n(1-\gamma)}$-concave on $E$.
		\item[(ii)] If $\frac{1}{1-\gamma} + K =0$, %in case (i)
		then the $\omega_{3}$ term disappears from \eqref{eq:conditionC}, the Gagliardo-Nirenberg inequality reduces to the Sobolev inequality,  $\omega_{1}$ and $\omega_{2}$ are equal up to a constant multiplicative factor and  $\frac{1-\gamma}{1-n(1-\gamma)}$-concave on $E$.
	\end{itemize}
In addition, in both cases {\rm (i)} and {\rm (ii)}, one has that $u(x)=A(\lambda+|x+x_0|^{p'})^\frac{1}{1-\alpha}$, $x\in E$, for some  $A,\lambda >0$ and $x_0\in -\overline E\cap \overline E$ with $\omega_1(x+x_0)=\omega_1(x)$, $x\in E$.
\end{theorem}

	We stress that the key step of the proof of Theorem~\ref{prop:equalitycase-intro} is tracing back the equalities inside the proof of Theorem~\ref{th:CKN2weights}. To this end it is crucial that Theorem~\ref{th:CKN2weights} has been proven directly with functions belonging to the correct Sobolev space $\dot W^{p,\alpha}(\omega_1, \omega_2; E)$: a proof with smooth and compactly supported functions would have let us obtain the same inequalities stated in Theorem~\ref{th:CKN2weights} via a density argument, but it would have not allowed to characterize the equality case. In sight of this, a generalized weighted integration by parts formula is proved in Proposition~\ref{prop:integrationbyparts}: we believe this technical instrument might deserve its own independent interests. 

The assumption concerning the attainability of the infimum in \eqref{C-1-2-constant-initial} is important in our argument; at this moment we have no certainty that this compactness property holds for two generic weights $\omega_1$ and $\omega_2$ which are related only by condition (C). On the other hand, when the weights are equal to each other (up to some multiplicative constant), the latter property follows by a duality argument.

In the  case $\gamma>1$ a ``dual'' Gagliardo-Nirenberg inequality to \eqref{eq:CKN2weights} can be established    via condition (C); to do this, we assume that $1+K(1-\gamma)>0$ and  by using the same notations from \eqref{L-M}, we introduce the constants 
\begin{equation}\label{C-1-1-constant-initial-dual}
	\tilde C_{K,L,M,C_0}=\left(\frac{1}{\gamma-1}-K\right)^{-\frac{M}{p}}C_0^{-L+(1-\gamma)n\left(\frac{M}{p}+L\right)}(\gamma \alpha)^{-L}(\gamma-1)^{-\frac{M}{p}-L}\frac{M^\frac{M}{p}(-L)^L}{(M+pL)^{\frac{M}{p}+L}},
\end{equation}
and 
\begin{equation}\label{C-1-2-constant-initial-dual}
	\tilde C_{\mathcal G_{\omega_1,\omega_2}}=\inf_{G\in \mathcal G_{\omega_1,\omega_2};\int_E G\omega_1=1}\frac{\displaystyle\left(\int_E G^\gamma(y)\omega_1^\frac{1}{p'}(y)\omega_2^\frac{1}{p}(y)dy\right)^{\frac{M}{p}+L}}{\displaystyle\left(\int_E G(y)|y|^{p'}\omega_1(y)dy\right)^\frac{L}{p'}}.
\end{equation}

\begin{theorem}\label{th:CKN2weights-2} $(\gamma>1)$ Let $n\geq 2$, $1<p< \infty$, $\gamma>1$ and $E\subseteq \mathbb R^n$ be an open convex cone, and  $\omega_i\colon E\to (0,+\infty)$ be homogeneous weights  with degree $\tau_i>-n$ and of class $\mathcal C^1$, $i\in \{1,2,3\}$, such that the triplet  $(\omega_1,\omega_2,\omega_3)$ satisfies  condition \textbf{\rm (C)} on $E$. 
	%	that satisfy \eqref{eq:conditionC} with parameters $p> 1$, $\gamma> \frac 1{p'}$, $\gamma\neq 1$,  $K \in \mathbb R$ and $C_0>0$.  
	Let $\alpha=\frac 1{p(\gamma-1)+1}$ and assume that $L<0$ and $1+K(1-\gamma)>0$. Then  $\tau_3=\frac{\tau_1}{p'}+\frac{\tau_2}p$ and  %(n+\tau_1)(\gamma-p\gamma+p-2)+(n+\tau_2)$, 
one has
\begin{equation}\label{eq:CKN2weightsgamma>1}
	\left(\int_E |u|^{\alpha p\gamma}\omega_3dx\right)^{\frac{1}{\alpha p\gamma}}\leq {\sf C}_2 \left(\int_E|\nabla u|^p\omega_2 dx\right)^{\frac{\theta_2}{p}}\left(\int_E|u|^{\alpha p }\omega_1 dx\right)^{\frac{1-\theta_2}{\alpha p }},\  \forall u\in \dot{W}^{p,\alpha}(\omega_1,\omega_2;E),
\end{equation} 
where 
 \begin{equation}\label{C-2-optimal}
	{\sf C}_2=\left(\tilde C_{K,L,M,C_0}\tilde C_{\mathcal G_{\omega_1,\omega_2}}\right)^\frac{1}{\alpha\gamma M} \  and\  \theta_2=-\frac{L}{\alpha\gamma M}\in (0,1).
\end{equation}
%
%
%
% $\theta_2=\frac{(n+\tau_1)\gamma-n-\tau_3}{(n+\tau_1)\gamma+\alpha\gamma(p-n-\tau_2)}$ \textbf{and $C_2>0$ is given by...} 
%	 \frac{(n+\tau_1)(\gamma-1)}{\alpha \gamma (p-\tau_2-n)+\gamma(n+\tau_1)}$.
\end{theorem}

%\textbf{ In (ii) however, the same term will be necessarily finite. }

Let us mention that it will follow from the proof of Theorem \ref{th:CKN2weights-2} that $M>0$ and thus the constant in \eqref{C-1-1-constant-initial-dual} is well defined. 

If the three weights are \textit{equal} and we take the limit $\gamma\to 1$ in the Gagliardo-Nirenberg inequalities (either in Theorem~\ref{th:CKN2weights} or Theorem~\ref{th:CKN2weights-2}), we are able in Theorem~\ref{log-Sobolev} to get sharp one-weighted $p$-log-Sobolev inequalities on $E$ together with the explicit sharp constant. However, this limiting method does not provide a way to characterize the equality cases. On the other hand, a direct OMT-based proof of the sharp one-weighted $p$-log-Sobolev inequality has been given recently in \cite{BDK} together with a full characterization of the equality cases. 

A further application of Theorem \ref{th:CKN2weights-2}, is  Theorem~\ref{th:faber-krahn},  a weighted Faber-Krahn inequality that is obtained by choosing the weights equal to each other and letting $\gamma\to +\infty$ in \eqref{eq:CKN2weightsgamma>1}.

%{\color{red} 

We finally underline how the characterization of the equality case for the weighted Gagliardo-Nirenberg inequality in the case $\gamma>1$ is still open. To the best of our knowledge, the answer to this question is still unknown even in the unweighted case, see e.g.\ the case $\alpha<1$ in \cite{CENV}. The main difficulty comes from the regularity of the transport map $\phi$: in the case $\gamma<1$ we are able to prove that the optimizer is strictly positive on compact subsets of its support and hence the singular part of the distributional Laplacian of $\phi$ vanishes. However, in the case $\gamma>1$, the optimizers are solution of a $p$-Laplace equation with additional singular terms, hence less regularity of its solutions has to be expected.
%}

The paper is organized as follows. In section \S \ref{section-2} we prove the main Gagliardo-Nirenberg inequalities with three weights, see Theorems \ref{th:CKN2weights} \& \ref{th:CKN2weights-2}. In section \S \ref{section-3}  we establish sharp  weighted Sobolev, Gagliardo-Nirenberg, log-Sobolev, Faber-Krahn and isoperimetric inequalities in the case when the weights are equal.  Section \S \ref{section-4} is devoted to the discussion of the rigidity, by proving Theorem \ref{prop:equalitycase-intro}. In section \S \ref{sec:Preliminar-0} we construct several classes of weights which verify our main condition (C) and prove the equivalence between \eqref{Lam-1-0} and the $\frac{1-\gamma}{1-n(1-\gamma)}$-concavity of the weight (both for $\gamma\neq 1$ and in the limit case $\gamma\to 1$).

\section{Weighted Gagliardo-Nirenberg inequalities: proof of Theorems \ref{th:CKN2weights} \& \ref{th:CKN2weights-2}}\label{section-2}

This section is devoted to the proofs of Theorems \ref{th:CKN2weights} and \ref{th:CKN2weights-2}. Before that we first discuss necessary conditions that the weights need to satisfy when condition (C) holds, by using the homogeneity property of the weights with respect to $y$ and $x$. We notice that by multiplying both sides of \eqref{eq:conditionC} by $\omega_1^{-1/{p}}\omega_2^{1/p}(x)$ we obtain the equivalent condition
\begin{equation}\label{eq:conditionCalt}
\left(\frac 1{1-\gamma}-n\right)\left(\frac{\omega_2^{\frac 1p}(y)\omega_1^{\frac1{p'}-\gamma}(y)}{\omega_2^{\frac np(1-\gamma)}(x)\omega_1^{\left(1-\frac np\right)(1-\gamma)}(x)}\right)^{\frac1{1-n(1-\gamma)}}\leq \left(\frac 1{1-\gamma}+K\right)\frac{\omega_3(x)}{\omega_1(x)}+C_0\frac{\nabla(\omega_2^{\frac 1p}\omega_1^{\frac{1}{p'}})(x)}{\omega_1(x)}\cdot y,
\end{equation}
  Here we focus on the homogeneity in $y$.

\begin{itemize}
	 
	\item {\it Homogeneity in $y$, case $\gamma<1$}. Replacing $y$ with $\lambda y$ in \eqref{eq:conditionCalt}, and letting $\lambda \to +\infty$ and $\lambda\to0$ we get
	\begin{equation}\label{homog-cond-1}
	0\leq\frac{\tau_2}p+\tau_1\left(\frac{1}{p'}-\gamma\right)\leq 1-n(1-\gamma).
	\end{equation}

	\item 	{\it Homogeneity in $y$, case $\gamma>1$}. In this case we can rearrange the terms in \eqref{eq:conditionCalt} in the following way:
	\begin{equation}\label{eq:changeorder}
	\left(\frac 1{\gamma-1}-K\right)\frac{\omega_3(x)}{\omega_1(x)}\leq \left(\frac1{\gamma-1}+n\right)\left(\frac{\omega_2^{\frac 1p}(y)\omega_1^{\frac1{p'}-\gamma}(y)}{\omega_2^{\frac np(1-\gamma)}(x)\omega_1^{\left(1-\frac np\right)(1-\gamma)}(x)}\right)^{\frac1{1-n(1-\gamma)}}+C_0\frac{\nabla(\omega_2^{\frac 1p}\omega_1^{\frac{1}{p'}})(x)}{\omega_1(x)}\cdot y.
	\end{equation}
	Since $\frac1{\gamma-1}-K>0$ by assumption  (see Theorem \ref{th:CKN2weights-2}), replacing $y$ with $\lambda y$ and letting $\lambda \to 0,$ we get
	\[
	\frac{\tau_2}p+\tau_1\left(\frac 1{p'}-\gamma\right)\leq 0.
	\]
	Moreover, choosing $y=x= \lambda v $ in \eqref{eq:conditionCalt} and letting $\lambda \to +\infty$ and $\lambda \to 0$ we obtain:
	$
	\tau_3=\frac {\tau_2}p+\frac{\tau_1}{p'}.
	$
	\item  In both cases, i.e., $\gamma<1$ and $\gamma>1$ (with  $\frac1{\gamma-1}-K>0$), choosing $y=\lambda x$ in \eqref{eq:conditionCalt} we obtain
	$
	\frac{\tau_2}p+\frac{\tau_1}{p'}\geq 0
	$
	and 
	\begin{equation}\label{gradient-nonnegative}
		{\nabla(\omega_2^{\frac 1p}\omega_1^{\frac{1}{p'}})(x)}\cdot y\geq 0
	\end{equation}
	for every $x,y\in E$.
\end{itemize}

The next lemma will be crucial in the sequel; its proof  relies on the AM-GM inequality, see \cite{Lam}.

\begin{lemma} \label{tr-det-inequality}
	Let $n\geq 2$ be a natural number,  $\gamma\geq 1-\frac{1}{n}$, $\gamma\neq 1$, $C>0$ and  $M$ be an $(n\times n)$-type positive-definite symmetric matrix. Then 
	\begin{equation}\label{tr-det}
		\frac{1}{1-\gamma}C^{1-n(1-\gamma)}({\rm det}(M))^{1-\gamma}\leq \left(\frac{1}{1-\gamma}-n\right)C+{\rm tr}(M).
	\end{equation}
	Moreover, the equality in \eqref{tr-det} holds if and only if $M=C I_n$. $($Hereafter, $I_n$ denotes the $(n\times n)$-type unitary matrix.$)$
\end{lemma}

The following integration by parts inequality will be crucial in the proofs of our main results. To formulate this result we denote by  $\Delta_{D'}\phi$ the distributional Laplacian of a convex function $\phi\colon E \to \mathbb R$ and by $\Delta_{A}\phi$  its absolutely continuous part. 

\begin{proposition}\label{prop:integrationbyparts}
	Let $p>1$,  $E\subseteq \mathbb R^n$ be an open convex cone, and  $\omega_i\colon E\to (0,+\infty)$ be weights 
	% with degree $\taf_i>-n$ and 
	of class $\mathcal C^1$, $i\in \{1,2\}$, such that $\nabla (\omega_1^{1/{p'}}\omega_2^{1/p}) (x)\cdot y\geq 0$ for every $x,y\in E.$
	Let   $\phi\colon \mathbb R^n\to \mathbb R$ be a convex function such that $\nabla\phi(x)\in \overline E$ for a.e. $x\in \overline E$. If $q>1$, and  $f\colon E\to [0,+\infty)$ is a measurable function such that  $f^{q-1}\nabla\phi\in L^{p'}(\omega_1;E)$ and  $\nabla f\in L^{p}(\omega_2;E)$, then 
	%	assume that for some $C_1>0$ one has
	%	\begin{equation}\label{M-A-like}
	%		u(x)^p\omega(x)= C_1e^{-|\nabla\phi(x)|^{p'}}\omega(\nabla\phi(x)) \det(D_A^2\phi(x))\ \ \text{for  a.e.}\ x\in E\cap {\rm supp}(u).
	%	\end{equation}
	\begin{equation}\label{eq:Byparts}
	\int_E f^{q}\omega_1^\frac{1}{p'}\omega_2^\frac{1}{p} \Delta_A\phi dx\leq  -q\int_E f^{q-1}\omega_1^\frac{1}{p'}\omega_2^\frac{1}{p} \nabla \phi \cdot \nabla f dx -\int_E f^{q} \nabla (\omega_1^\frac{1}{p'}\omega_2^\frac{1}{p}) \cdot \nabla \phi dx.
	\end{equation}
	
\end{proposition}
\begin{proof} The proof is adapted  from Cordero-Erausquin, Nazaret and Villani \cite{CENV} (for the unweighted case) and Balogh, Don and Krist\'aly \cite{BDK} (involving  only one weight and  slightly different integrability conditions).  
	Let us denote by $S\coloneqq  {\rm int}\{x\in \mathbb R^n: \phi(x)<+\infty\}$; due to the assumption, $E\cap S$ is open, convex and contains the support of $u$ except, possibly, for a Lebesgue null set.
	
	For every $k\geq 1$, let $\theta_k^0\colon[0,+\infty)\to [0,1]$ be a non-decreasing $\mathcal C^\infty$ function such that $\theta_k^0(s)= 0$ for $s\leq \frac{1}{2k}$ and $\theta_k^0(s)= 1$ for $s\geq \frac{1}{k}$; in addition, let $\theta_k\colon E\to [0,1]$  be the locally Lipschitz function $\theta_k(x)=\theta_k^0(d(x,\partial E))$, $x\in E.$ 
	Let $f_k=f\theta_k$, $k\geq 1$. It turns out that $\nabla f_k\in L^p(\omega_2;E)$ for every $k\geq 1$ and ${\rm supp}f_k\subseteq S_k$ where $S_k=\{x\in E:d(x,\partial E)\geq \frac{1}{2k}\}.$  Furthermore, for every fixed $k\geq 1$, by a trivial extension, one has that  $\nabla f_k\in L^p(\mathbb R^n)$.

	%For every $k\geq 1$, let $\theta_k:[0,\infty)\to [0,1]$ be a locally Lipschitz function such that $\theta_k(x)= 0$ for $d(x,\partial (E\cap S))\leq \frac{1}{2k}$, $\theta_k(x)= 1$ for $d(x,\partial(E\cap S))\geq \frac{1}{k}$ and $\theta_k(x)=\theta_k^0(d(x,\partial (E\cap S)))$ for $\frac{1}{2k}< d(x,\partial (E\cap S))< \frac{1}{k}$, where $\theta_k^0:(\frac{1}{2k},\frac{1}{k})\to (0,1)$ is an increasing $\mathcal C^\infty$ function. 
	%
	% We now consider the function $f_k=u\theta_k$, $k\geq 1$. A simple argument shows that $f_k\in W^{1,p}(\omega;E)$ for every $k\geq 1$ and ${\rm supp}f_k\subseteq S_k$ where $S_k=\{x\in E\cap S:d(x,\partial (E\cap S))\geq \frac{1}{2k}\}.$  Furthermore, for every fixed $k\geq 1$, the function  $f_k$ belongs to the usual (unweighted) Sobolev space $W^{1,p}(E)$ as well. 

	%
	% $[0,1]$ and $\theta\equiv 1$ on $[2,\infty).$  Let $\theta_k(x)=\theta(\omega(kx))$, $x\in E$ and $k\geq 1$. 
	
	% To see this, we observe that ${\rm supp}f_k\subseteq S_k$ where $S_k=\{x\in E:\omega(x)\geq \frac{1}{k^\tau}\}.$ Then, for every fixed $k\geq 1$, we have that
	%$$\int_Ef_k^pdx\leq \int_{S_k}f_k^pdx\leq k^\tau\int_{S_k}f_k^p\omega dx\leq k^\tau\int_{E}f_k^p\omega dx\leq k^\tau\int_{E}u^p\omega dx<+\infty,$$
	%and in a similar way
	%$$\int_E|\nabla f_k|^pdx\leq \int_{S_k}|\nabla f_k|^pdx\leq 2^pk^\tau\left(\int_{E}|\nabla u|^p\omega dx+C^p\int_{E}u^p\omega dx\right) <+\infty,$$
	%where $C=\max\{|\nabla \theta(s)|:s\in [1,2]\}.$

	Let us consider a cut-off function $\chi\colon \mathbb R^n\to [0,1]$ of class $\mathcal C^\infty$ such that $\chi\equiv 1$ on $B=B(0,1)$ and $\chi\equiv 0$ on $\mathbb R^n\setminus B(0,2)$. 
	Fix $x_0\in E$ and $k\geq 1$. For 	$\varepsilon\in(0,1)$  small enough, we define 
	\[
	f_{k,\varepsilon}(x)=\min\left\{f_k\left(x_0+\frac{x-x_0}{1-\varepsilon}\right),\chi(\varepsilon x)f_k(x)\right\},\  \ x\in E.
	\]
	Then $f_{k,\varepsilon}\geq 0$ has compact support in $E\cap S$ and since $|\nabla\chi|\leq C_0$ for some $C_0>0$, one has that $\nabla f_{k,\varepsilon} \in L^p(\omega_2;E)\cap L^p(\mathbb R^n)$. 
	
	Given  a non-negative function $\kappa\in C_c^\infty(B)$  such that $\displaystyle\int_{B} \kappa  dx=1$, we consider the convolution kernel $\kappa_\delta(x)=\frac{1}{\delta^{n}}\kappa( x/\delta)$, $\delta>0,$  and define the convolution function  $$f_{k,\varepsilon}^\delta(x)=(f_{k,\varepsilon}\star \kappa_\delta)(x)\coloneqq \int_E f_{k,\varepsilon}(y)\kappa_\delta(x-y) dy\geq 0.$$
	%	SEBASTIANO:	Given $\kappa\in C_c^\infty(B(0,1))$ be a positive function with $\int \kappa \omega dx=1$, we consider the convolution kernel $\kappa_\delta(x)=\frac{1}{\delta^{n+\tau}}\kappa( x/\delta)$ and define the convolution function  $$f_{k,\varepsilon}^\delta(x)=(f_{k,\varepsilon}\star \kappa_\delta)(x)\coloneqq \int_E f_{k,\varepsilon}(y)\kappa_\delta(x-y)\omega(y) dy\geq 0.$$ 
	It is standard that the functions $f_{k,\varepsilon}^\delta$ are smooth on $E$ for every  $\varepsilon\in (0,1)$ belonging to the above range, and $\delta>0$. Moreover, since $\phi$ is convex, and ${\rm supp} f_{k,\varepsilon}^\delta$ is compact in $E$, then $\nabla\phi$ is essentially bounded on ${\rm supp}(f_{k,\varepsilon}^\delta)$ for every sufficiently small $\delta>0$.
	
	The idea is to apply the usual divergence theorem to the smooth function $f_{k,\varepsilon}^\delta$  and pass to the limit as $\delta\to 0$, $\varepsilon \to 0$ and $k\to \infty$, in this order, obtaining the expected inequality \eqref{eq:Byparts}.  
	
	Taking into account that $f_{k,\varepsilon}^\delta=0$ on $\partial E$, and $\Delta_A\phi\leq \Delta_{\mathcal D'}\phi$ in the distributional sense, the divergence theorem implies 
	%	\begin{equation}
	\begin{eqnarray}\label{eq:GaussGreen}
	\nonumber	\int_E (f_{k,\varepsilon}^\delta)^{q}\omega_1^\frac{1}{p'}\omega_2^\frac{1}{p} \Delta_A\phi dx&\leq& \int_E (f_{k,\varepsilon}^\delta)^{q}\omega_1^\frac{1}{p'}\omega_2^\frac{1}{p} \Delta_{\mathcal D'}\phi dx\\&=& -q\int_E (f_{k,\varepsilon}^\delta)^{q-1}\nabla f_{k,\varepsilon}^\delta\cdot \nabla \phi \omega_1^\frac{1}{p'}\omega_2^\frac{1}{p} dx	\nonumber \\&&-\int_E (f_{k,\varepsilon}^\delta)^{q}\nabla (\omega_1^\frac{1}{p'}\omega_2^\frac{1}{p}) \cdot \nabla \phi dx.
	\end{eqnarray}
	%	\end{equation}
	Since $\phi$ is convex, thus $\Delta_A\phi\geq 0$ on $E$, and $\nabla (\omega_1^\frac{1}{p'}\omega_2^\frac{1}{p}) (x)\cdot \nabla\phi(x)\geq 0$ for a.e.  $x\in E$, where we used the assumption that $\nabla\phi(x)\in \overline E$ for a.e. $x\in \overline E$, it turns out that
	\begin{equation}\label{I-positive}
	I_{\omega_1,\omega_2,\phi}:=\omega_1^\frac{1}{p'}\omega_2^\frac{1}{p} \Delta_A \phi + \nabla (\omega_1^\frac{1}{p'}\omega_2^\frac{1}{p}) \cdot \nabla \phi \geq 0\ \ {\rm a.e.\ on}\ E.
	\end{equation}
	Thus, inequality \eqref{eq:GaussGreen} reads as
	\begin{equation}\label{eq:GaussGreen2}
	\int_E (f_{k,\varepsilon}^\delta)^{q} I_{\omega_1,\omega_2,\phi}(x)  dx\leq-q\int_E (f_{k,\varepsilon}^\delta)^{q-1}\nabla f_{k,\varepsilon}^\delta\cdot \nabla \phi \omega_1^\frac{1}{p'}\omega_2^\frac{1}{p} dx.
	\end{equation}
	
	Let us observe first that 
	\[
	\int_E (f_{k,\varepsilon}^\delta)^{q-1}\nabla f_{k,\varepsilon}^\delta\cdot \nabla \phi \omega_1^\frac{1}{p'}\omega_2^\frac{1}{p} dx\rightarrow \int_E f_{k,\varepsilon}^{q-1}\nabla f_{k,\varepsilon}\cdot \nabla \phi \omega_1^\frac{1}{p'}\omega_2^\frac{1}{p} dx  \quad \text{as $\delta\to0,$}\]
	which is due to the fact that 
	$\nabla \phi$ is essentially bounded on ${\rm supp}(f_{k,\varepsilon}^\delta)$, and  $(f_{k,\varepsilon}^\delta)^{q-1}\nabla \phi$ converges in $L^{p'}(\omega_1;E)$ to $(f_{k,\varepsilon})^{q-1}\nabla \phi$, and $\nabla f_{k,\varepsilon}^\delta$ converges in $L^p(\omega_2;E)$ to $\nabla f_{k,\varepsilon}$ as $\delta\to 0$, coming from properties of convolution and superposition operators based on the dominated convergence theorem. 
	%	Since $\phi$ is convex, then $\Delta_A\phi\geq 0$; moreover, by Proposition \ref{prop-log-concave} we also have that   $\nabla\phi\cdot \nabla\omega\geq 0$ a.e.\ on $E$, since $\nabla\phi(x)\in \overline E$ for a.e. $x\in \overline E$. 
	Thus, letting $\delta\to 0$ in \eqref{eq:GaussGreen2}, the latter properties together with  Fatou's lemma and  \eqref{I-positive} imply  that
	\begin{equation}\label{eq:GaussGreen3}
	\int_E f_{k,\varepsilon}^{q} I_{\omega_1,\omega_2,\phi}(x)  dx\leq-q\int_E f_{k,\varepsilon}^{q-1}\nabla f_{k,\varepsilon}\cdot \nabla \phi \omega_1^\frac{1}{p'}\omega_2^\frac{1}{p} dx.
	\end{equation}
	
	The next step is to take the limit in \eqref{eq:GaussGreen3} whenever $\varepsilon\to 0.$ By the definition of $f_{k,\varepsilon}$, we have that, passing to a subsequence of $\varepsilon=(\varepsilon_{k,l})_{l\in \mathbb N}$, the sequence $f_{k,\varepsilon}$ converges to $f_k$ a.e.\ as $\varepsilon\to 0.$ 
	%
	%  $f_{k,\varepsilon}(x)\coloneqq  f_k\left(x_0+\frac{x-x_0}{1-\varepsilon}\right)$ converges a.e. to $f_k$ as  $\varepsilon\to 0;$ this fact directly follows by the definition of   
	%To see this, we first observe that if $\varphi$ is any compactly supported test-function, it follows that $$\int_E f_{k,\varepsilon}(x)\varphi(x)dx=(1-\varepsilon)^n\int_{x_0+\frac{E-x_0}{1-\varepsilon}} f_{k}(y)\varphi(\varepsilon x_0+(1-\varepsilon)y)dy\rightarrow\int_E f_{k}(x)\varphi(x)dx$$
	% as $\varepsilon\to 0,$ i.e., $f_{k,\varepsilon}$ converges to $f_k$ in the sense of distributions. In addition, a simple change of variables implies that $f_{k,\varepsilon}$ is bounded in $W^{1,p}(E)$ as $\varepsilon \to 0.$ These properties imply that $f_{k,\varepsilon}$ converges weakly to $f_k$ in $W^{1,p}(E)$ as $\varepsilon\to 0.$ Moreover, by the Rellich-Kondrachov embedding theorem, we also have that -- up to a subsequence -- the sequence $f_{k,\varepsilon}$ converges locally strongly to $f_k$ in $L^p(E)$. 
	%% ; in fact, in $L^q(E)$ for every $q\in (1,pn/(n-p))$ whenever $p<n$, in  $L^q(E)$ for every $q\in (1,\infty)$ whenever $p=n$, and in $L^q(E)$ for every $q\in (1,\infty]$ whenever $p>n$. 
	% In particular, $f_{k,\varepsilon}$ converges a.e. to $f_k$, which concludes the proof of the claim. 
	%
	%	By relation \eqref{M-A-like} and a change of variable we have that . Moreover, 
	In addition, by construction we have that $ f_{k,\varepsilon}\leq f_k\leq u$ and  by assumption $f^{q-1}\nabla\phi\in L^{p'}(\omega_1;E)$;  accordingly,  the dominated convergence theorem implies  that
	\[
	f_{k,\varepsilon}^{q-1}\nabla\phi\to f_{k}^{q-1}\nabla\phi \quad \text{in $L^{p'}(\omega_1;E)$ as $\varepsilon\to 0$}.
	\]
	One can also prove that $\nabla 	f_{k,\varepsilon}$ converges to $\nabla 	f_{k}$ in the sense of distributions and $\nabla 	f_{k,\varepsilon}$ is uniformly bounded in $L^p(E)$ w.r.t.\ $\varepsilon>0;$ hence, $\nabla 	f_{k,\varepsilon}$ converges weakly in $L^p(E)$ to $\nabla f_k$ as $\varepsilon\to 0.$
	Combining the last two facts, it turns out that
	\[
	\int_E f_{k,\varepsilon}^{q-1}\nabla f_{k,\varepsilon}\cdot \nabla \phi \omega_1^\frac{1}{p'}\omega_2^\frac{1}{p} dx \to \int_E f_{k}^{q-1}\nabla f_{k}\cdot \nabla \phi \omega_1^\frac{1}{p'}\omega_2^\frac{1}{p} dx \quad \text{as $\varepsilon\to 0$}.
	\]
	By Fatou's lemma, once we let $\varepsilon\to 0$ in \eqref{eq:GaussGreen3}, we obtain that
	\begin{equation}\label{final-1}
	\int_E f_{k}^{q} I_{\omega_1,\omega_2,\phi}(x)  dx\leq-q\int_E f_{k}^{q-1}\nabla f_{k}\cdot \nabla \phi \omega_1^\frac{1}{p'}\omega_2^\frac{1}{p} dx.
	\end{equation}
	
	As a final step, we take the limit $k\to \infty$ in \eqref{final-1}. We observe that since $x\mapsto d(x,\partial E)$ is locally Lipschitz, it is differentiable a.e.\ in $E$; furthermore, for a.e.\ $x\in E$ one has that $\nabla d(\cdot,\partial E)(x)=-{\bf n}(x^*)$, where $x^*\in \partial E$ is the unique point with  $|x-x^*|=d(x,\partial E)$ and $\textbf{n}(x^*)$ is the unit outer normal vector at $x^*\in \partial E$. Since $E$ is convex and $\nabla\phi(x)\in \overline E$ for a.e.\ $x\in \overline E$, it turns out that 
	$\textbf{n}(x^*)\cdot \nabla\phi(x)\leq 0$ for a.e. $x\in E$. In particular,  the monotonicity of $\theta_k^0$ implies that for a.e. $x\in E$ 
	%with  $\frac{1}{2k}\leq d(x,\partial E)\leq \frac{1}{k}$ 
	one has 
	\begin{eqnarray*}
		\nabla f_{k}(x)\cdot \nabla \phi(x)&=&\theta_k(x)\nabla f(x)\cdot \nabla \phi(x)-f(x)(\theta_k^0)'(d(x,\partial E))\textbf{n}(x^*)\cdot \nabla \phi(x)\\&\geq &\theta_k(x)\nabla f(x)\cdot \nabla \phi(x).
	\end{eqnarray*}
	%		Due to the fact that $\nabla\theta_k(x)=0$ for  $d(x,\partial E)\in [0,\frac{1}{2k}]\cup [\frac{1}{k},\infty)$, 
	The above arguments together with 		
	%		
	%		Now, by Lemma \ref{lemma-log-concave} and the fact  that $\nabla\phi(x)\in \overline E$ for a.e.\ $x\in \overline E$, the monotonicity of $\theta$ implies that 
	\eqref{final-1} imply
	\[
		\int_E \theta_k^{q} f^{q} I_{\omega_1,\omega_2,\phi}(x)  dx\leq-q\int_E \theta_k^{q}f^{q-1}\nabla f\cdot \nabla \phi \omega_1^\frac{1}{p'}\omega_2^\frac{1}{p} dx.
	\]
	Letting $k\to \infty$ and taking into account that $\theta_k\to 1$ for a.e. $x\in E$ as $k\to \infty$,  Fatou's lemma  and the dominated convergence theorem imply the  inequality \eqref{eq:Byparts}. 
\end{proof}

\begin{remark}\rm \label{remark-appendix}
	When we apply Proposition \ref{prop:integrationbyparts} in the proof of Theorems \ref{th:CKN2weights} and \ref{th:CKN2weights-2} we shall choose $q= \alpha p \gamma = \frac{p \gamma}{p\gamma- p +1}>1$. From the proof of Theorems \ref{th:CKN2weights} and \ref{th:CKN2weights-2} it will be clear that the integrability assumptions $f^{\alpha p\gamma-1}\nabla\phi\in L^{p'}(\omega_1;E)$ and  $\nabla f\in L^{p}(\omega_2;E)$ are satisfied and thus Proposition 
	\ref{prop:integrationbyparts} applies.  We notice that both terms 
	$	\displaystyle\int_E f^{\alpha p\gamma}\omega_1^\frac{1}{p'}\omega_2^\frac{1}{p} \Delta_A\phi dx$ and  $\displaystyle\int_E f^{\alpha p\gamma} \nabla (\omega_1^\frac{1}{p'}\omega_2^\frac{1}{p}) \cdot \nabla \phi dx$ are finite. Indeed, both of them are non-negative, and due to inequality  \eqref{eq:Byparts},  their sum is bounded from above by $$\alpha p\gamma\left|\int_E f^{\alpha p\gamma-1}\omega_1^\frac{1}{p'}\omega_2^\frac{1}{p} \nabla \phi \cdot \nabla f dx\right|\leq \alpha p\gamma\left(\int_E f^{(\alpha p\gamma-1)p'}|\nabla \phi|^{p'}\omega_1 dx\right)^\frac{1}{p'}\left(\int_E |\nabla f|^p\omega_2\right)^\frac{1}{p}<+\infty.$$
	\end{remark}

After this preparatory part, we focus to the proof of Theorems \ref{th:CKN2weights} and \ref{th:CKN2weights-2}. 
As mentioned in the introduction the strategy of the proof is based on optimal mass transport theory. For more details on this method we refer to \cite{Villani}. To start the proof, we assume that $\gamma\neq 1$ and  we consider the space of functions 
\[
\mathcal G_{\omega_1,\omega_2}\coloneqq\left\{G\colon E\to [0,+\infty): G\in L^1(\omega_1;E)\cap L^\gamma(\omega_1^\frac{1}{p'}\omega_2^\frac{1}{p};E),\ \int_E G(y)|y|^{p'}\omega_1(y)dy<+\infty\right\}.
\]
Clearly, non-negative functions of $C_0^\infty(\mathbb R^n)$ belong to $ \mathcal G_{\omega_1,\omega_2}.$ 

We consider $u\in \dot{W}^{p,\alpha}(\omega_1,\omega_2;E)\setminus \{0\}$ and let 
	\begin{equation}\label{f-u}
	f\coloneqq\frac {|u|}{\left(\displaystyle \int_E |u|^{\alpha p}\omega_1 dx\right)^{\frac{1}{\alpha p}}}\geq 0.
\end{equation}
	If $\displaystyle \int_E|u|^{\alpha p \gamma}\omega_3 dx=+\infty$, we have nothing to prove in \eqref{eq:CKN2weights}; thus, we may consider without loss of generality that it is finite, i.e., $\displaystyle \int_E f^{\alpha p \gamma}\omega_3 dx<+\infty$; for \eqref{eq:CKN2weightsgamma>1} this fact will follow as we shall see in the sequel. 
	Let $G\in \mathcal G_{\omega_1,\omega_2}$ be  a function  such that
	\[
	\int_E G \omega_1dx=\int_E f^{\alpha p}\omega_1 dx=1.
	\]
	Let $U\coloneqq {\rm supp} u\subset \mathbb R^n.$ We can assume that $E\cap U\neq \emptyset$; otherwise, both \eqref{eq:CKN2weights} and \eqref{eq:CKN2weightsgamma>1} become trivial.   By Brenier's theorem (see e.g.\ \cite{VillaniBook}) we can find a convex function $\phi\colon E \to \mathbb R$ such that the Monge-Amp\`ere equation holds, i.e., 
	\begin{equation}\label{Monge-Ampere}
			f^{\alpha p}(x)\omega_1(x)= G(\nabla\phi(x))\omega_1(\nabla\phi(x))\det(D_A^2\phi(x))\ \ {\rm a.e.}\ x\in E\cap U,
	\end{equation}
	or equivalently,
	\begin{equation}\label{eq:MAgamma}
	\frac{1}{1-\gamma}G^{\gamma-1}(\nabla\phi(x))\omega_1^{\gamma-1}(\nabla\phi (x))=f^{\alpha p(\gamma-1)}(x)\omega_1^{\gamma-1}(x) \frac{\det^{1-\gamma} (D_A^2\phi(x))}{1-\gamma}\ \ {\rm a.e.}\ x\in E\cap U,
	\end{equation}
where $D_A^2\phi$ denotes the Hessian of $\phi$ in the sense of Alexandrov (being the absolutely continuous part of the distributional Hessian of $\phi$), see Villani \cite{Villani}. In a similar way, let $\Delta_A\phi={\rm tr}D_A^2\phi$ be the Laplacian of $\phi$. 

	Multiplying both sides of \eqref{eq:MAgamma} by $\omega_1^{\frac1{p'}-\gamma}(\nabla\phi(x)) \omega_2^{\frac 1p}(\nabla\phi(x))$ and integrating with respect to the measure $f^{\alpha p}(x)\omega_1(x) dx=G(\nabla\phi(x))\omega_1(\nabla\phi(x))\det(D_A^2\phi(x))dx$ we obtain
	\begin{equation}\label{eq:MAintegrated}
	\begin{aligned}
	\frac {C_0^{(1-\gamma)n}}{1-\gamma}\int_{E\cap U}&G^\gamma(\nabla \phi(x))\omega_1^{\frac1{p'}}(\nabla\phi (x))\omega_2^{\frac 1p}(\nabla\phi(x))\det (D_A^2\phi(x))dx\\&=\frac{1}{1-\gamma}\int_{E\cap U} f^{\alpha p\gamma}(x) \omega_1(x)^\gamma {\rm det} ^{1-\gamma}(C_0D_A^2\phi(x))\omega_1(\nabla\phi (x))^{\frac{1}{p'}-\gamma}\omega_2(\nabla\phi(x))^{\frac 1p}dx.
	\end{aligned}
	\end{equation}
	Since $G\in  \mathcal G_{\omega_1,\omega_2}$,  a change of variables shows that  the left-hand side is integrable; thus the right-hand side (denoted by RHS) is also integrable. 
By  Lemma \ref{tr-det-inequality}, for  a.e.\ $x\in {E\cap U}$	we obtain that
	\begin{equation}\label{eq:from21toC}
	\begin{aligned}
	&\frac1{1-\gamma}\frac{\omega_2(\nabla\phi(x))^{\frac 1p}\omega_1(\nabla\phi(x))^{\frac 1{p'}-\gamma}}{\omega_1^{1-\gamma}(x)}\cdot \det(C_0D_A^2\phi(x))^{1-\gamma}\\=&\frac 1{1-\gamma}\left[\left(\frac{\omega_2(\nabla\phi(x))^{\frac 1p}\omega_1(\nabla\phi(x))^{\frac 1{p'}-\gamma}}{\omega_2(x)^{\frac np(1-\gamma)}\omega_1(x)^{\left(1-\frac np\right)(1-\gamma)}}\right)^{\frac1{1-n(1-\gamma)}}\right]^{1-n(1-\gamma)}\left[\det\left(C_0\frac{\omega_2(x)^{\frac 1p}}{\omega_1(x)^{\frac 1p}}D_A^2\phi(x)\right)\right]^{1-\gamma}\\
	\leq&\left(\frac1{1-\gamma}-n\right)\left(\frac{\omega_2(\nabla\phi(x))^{\frac 1p}\omega_1(\nabla\phi(x))^{\frac 1{p'}-\gamma}}{\omega_2(x)^{\frac np(1-\gamma)}\omega_1(x)^{\left(1-\frac np\right)(1-\gamma)}}\right)^{\frac1{1-n(1-\gamma)}}+C_0\left(\frac{\omega_2(x)}{\omega_1(x)}\right)^{\frac 1p}\Delta_A\phi(x)\\
	\leq&
	\left(\frac{1}{1-\gamma}+K\right)\frac{\omega_3(x)}{\omega_1(x)}+C_0\frac{\nabla\left(\omega_2^{\frac 1p}\omega_1^{\frac 1{p'}}\right)(x)}{\omega_1(x)}\cdot \nabla\phi(x)+C_0\left(\frac{\omega_2(x)}{\omega_1(x)}\right)^{\frac 1p}\Delta_A\phi(x),
	\end{aligned}
	\end{equation}
	where in the last inequality we used condition (C) with $y=\nabla\phi(x)\in E$, together with the fact that $\nabla\phi(E)\subseteq E$. 
		Using \eqref{eq:from21toC} we can estimate the right-hand side of \eqref{eq:MAintegrated}, to infer that  
	\begin{equation}\label{eq:stimaRHS}
	\begin{aligned}
	{\rm RHS}\leq&\left(\frac{1}{1-\gamma}+K\right) \int_{E}f^{\alpha p\gamma}(x) \omega_3(x)dx+C_0\int_{E\cap U} f^{\alpha p\gamma}(x)\nabla(\omega_2^{\frac 1p}\omega_1^{\frac 1{p'}})(x)\cdot \nabla \phi(x) dx\\ &+C_0\int_{E} f^{\alpha p\gamma}(x) \left(\frac{\omega_2(x)}{\omega_1(x)}\right)^{\frac 1p}\Delta_{ A}\phi(x)\omega_1(x)dx.
	\end{aligned}
	\end{equation}
	Note that since $f\in \dot{W}^{p,\alpha}(\omega_1,\omega_2;E)\setminus \{0\}$, we have that $\nabla f\in L^p(\omega_2;E)$, while from the Monge-Amp\`ere equation \eqref{Monge-Ampere} and a change of variables together with $(\alpha p\gamma -1)p'=\alpha p$ we obtain that 
$f^{\alpha p\gamma-1}\nabla\phi\in L^{p'}(\omega_1;E)$, since 
\begin{equation}\label{equation-inter-1}
	\int_E f^{(\alpha p\gamma-1)p'}(x)|\nabla\phi(x)|^{p'}\omega_1(x)dx=\int_E G(y)|y|^{p'}\omega_1(y)dy<+\infty,
\end{equation}
where in the last step we used the fact that $G\in \mathcal G_{\omega_1,\omega_2}.$ Since condition (C) implies that $$
	{\nabla(\omega_2^{\frac 1p}\omega_1^{\frac{1}{p'}})(x)}\cdot y\geq 0,\ \ x,y\in E,$$
 see  \eqref{gradient-nonnegative}, we are in the position to apply Proposition \ref{prop:integrationbyparts} with $q=\alpha p \gamma$ (see also Remark \ref{remark-appendix}), obtaining that
%	
%	 Note that since $\nabla\phi(\overline E)\subseteq \overline E$ and  $E$ is a convex cone, we have that $\nabla\phi(x)\cdot  \textbf{n}(x)\leq 0$ for $x\in \partial E$. Therefore, since $\partial(E\cap U)\subseteq \partial E \cup \partial U$, the last term in \eqref{eq:stimaRHS} is non-positive, being zero for $x\in \partial U$ and  non-positive for $x\in \partial E$. Accordingly, we can drop the last term from \eqref{eq:stimaRHS}, obtaining 
	\begin{equation}\label{eq:stimaRHS-2}
		\begin{aligned}
			{\rm RHS}\leq&  \left(\frac{1}{1-\gamma}+K\right) \int_Ef^{\alpha p\gamma} \omega_3(x)dx-C_0\alpha p\gamma\int_E f^{\alpha p\gamma-1}\omega_1^\frac{1}{p'}\omega_2^\frac{1}{p} \nabla \phi \cdot \nabla f dx.
		\end{aligned}
	\end{equation}
%	
%	
%
%
%In the above arguments we choose  $F:=f^{\alpha p},$ where $\alpha=\frac 1{p(\gamma-1)+1}>0$.  
	By  Young's inequality we can write for any $\mu>0$ that 
	\begin{equation}\label{eq:CauchySchwartz}
	\begin{aligned}
	& -\alpha p\int_E f^{\alpha p \gamma-1}\nabla f\cdot \nabla \phi \,\omega_2^{\frac 1p} \omega_1^{\frac 1{p'}}dx
%	& \hphantom{\int_E F^{\gamma-1}\nabla F\cdot \nabla \phi\omega_2^{\frac 1{p}} \omega_1^{\frac 1{p'}}dx}\leq \alpha p\int_E f^{\alpha p\gamma-1}|\nabla f||\nabla \phi|\, \omega_2^{\frac 1p} \omega_1^{\frac 1{p'}}dx\\
	 \leq \frac{\alpha}{\mu^p}\int_E |\nabla f|^p\omega_2dx+\frac{\alpha p \mu^{p'}}{p'}\int_E f^{(\alpha p\gamma -1)p'} |\nabla \phi|^{p'} \omega_1dx.
	\end{aligned}
	\end{equation}
	
	 By the Monge-Amp\`ere equation \eqref{Monge-Ampere}, relations \eqref{eq:MAintegrated} and  \eqref{equation-inter-1}, and inequalities  \eqref{eq:stimaRHS-2} and \eqref{eq:CauchySchwartz}, we have that 
%	 
%	\begin{equation}\label{eq:2.6}
%	\begin{aligned}
%-\int_E F^{\gamma-1}\nabla F\cdot \nabla \phi\,\omega_2^{\frac 1{p}} \omega_1^{\frac 1{p'}}dx &\leq \frac{\alpha}{\mu^p}\int_E |\nabla f|^p\omega_2dx+\alpha (p-1)\mu^{p'}\int_E f^{\alpha p} |\nabla \phi|^{p'} \omega_1dx\\ &=\frac{\alpha}{\mu^p}\int_E |\nabla f|^p\omega_2dx+\alpha (p-1)\mu^{p'} \int_E G(\nabla \phi)\det D_A^2\phi|\nabla \phi|^{p'}\omega_1(\nabla\phi)dx\\&=\frac{\alpha}{\mu^p}\int_E |\nabla f|^p\omega_2dx+\alpha(p-1)\mu^{p'}\int_E G(y)|y|^{p'}\omega_1(y)dy.
%	\end{aligned}
%	\end{equation}
	\begin{equation}\label{eq:2.7}
	\begin{aligned}
	\frac {C_0^{(1-\gamma)n}}{1-\gamma}&\int_E G^\gamma(y)\omega_1^{\frac 1{p'}}(y)\omega_2^{\frac1{p}}(y)dy-\frac{C_0\gamma(p-1)\mu^{p'}}{p(\gamma-1)+1}\int_E G(y)|y|^{p'}\omega_1(y)dy\\ &\leq
	\left(\frac 1{1-\gamma}+K\right)\int_E f^{\alpha p \gamma}\omega_3 dx+\frac{C_0\gamma}{\mu^{p}(p(\gamma-1)+1)}\int_E|\nabla f|^p\omega_2dx.
	\end{aligned}
	\end{equation}

In what follows we shall distinguish two cases according to $\gamma < 1$ and $\gamma >1$.\\ 

\noindent \textbf{Proof of Theorem \ref{th:CKN2weights}}
 ($\gamma<1$).  Let us define 
\begin{eqnarray}\label{K-1-definition}
	\nonumber	K_1&\coloneqq&K_1(C_0,\gamma,\omega_1,\omega_2,\mu)\\&=&\frac {C_0^{(1-\gamma)n}}{1-\gamma}\int_E G^\gamma(y)\omega_1^{\frac 1{p'}}(y)\omega_2^{\frac 1p}(y)dy-\frac{C_0\gamma(p-1)\mu^{p'}}{p(\gamma-1)+1}\int_E G(y)|y|^{p'}\omega_1(y)dy.
\end{eqnarray}
%where the supremum is taken over $G\in\mathcal G_{\omega_1,\omega_2}$ with $G\geq 0$ and $\int_E G\omega_1 dx=1$.
 We notice that for every sufficiently small $\mu>0$, we have $K_1>0$; we fix such a value of $\mu>0$.   
% Moreover, with the choice \eqref{mu-choice} for $\mu$, it turns out that $K_1>0$. 
Inequality \eqref{eq:2.7} implies that
	\begin{equation}\label{eq:beforeoptimum}
	K_1\leq\left(\frac 1{1-\gamma}+K\right)\int_E f^{\alpha p \gamma}\omega_3 dx+\frac{C_0\gamma}{\mu^{p}(p(\gamma-1)+1)}\int_E|\nabla f|^p\omega_2dx.
	\end{equation}
 Using \eqref{f-u}, inequality \eqref{eq:beforeoptimum} becomes
	\begin{equation}\label{eq:renormalized}
	K_1\leq K_2\frac{\displaystyle\int_E |u|^{\alpha p \gamma}\omega_3 dx}{\left(\displaystyle\int_E |u|^{\alpha p }\omega_1 dx\right)^{\gamma}}+K_3\frac{\displaystyle\int_E|\nabla u|^p\omega_2dx}{\left(\displaystyle\int_E |u|^{\alpha p} \omega_1 dx\right)^{\frac 1\alpha}},
	\end{equation}
	where we have set 
	\begin{equation}\label{K2-K3}
			K_2\coloneqq \frac 1{1-\gamma}+K\ \ {\rm and} \ \ K_3\coloneqq\frac{C_0\gamma}{\mu^{p}(p(\gamma-1)+1)}>0.
	\end{equation}
	 Note that since $\gamma < 1$ we have  $K_2\geq 0$; indeed, if $y\to 0$ in the condition (C), the latter property immediately follows.

	To continue the proof we apply a rescaling of the function $u$ and we shall optimise in terms of the rescaling parameter: we shall replace $u$ by $u_\lambda(x)\coloneqq \lambda^{\frac{n+\tau_1}{\alpha p}}u(\lambda x)$, $\lambda>0$. By a standard change of variable one can verify that $\displaystyle\int_E |u_\lambda|^{\alpha p} \omega_1dx=\int_E |u|^{\alpha p}\omega_1 dx$ and equation \eqref{eq:renormalized} is then equivalent to
	\begin{equation}\label{eq:f_lambda}
	K_1\leq K_2 \lambda^{-L}\frac{\|u\|^{\alpha p\gamma}_{\omega_3,\alpha p \gamma}}{\|u\|^{\alpha p\gamma}_{\omega_1,\alpha p }}+K_3\lambda^M\frac{\|\nabla u\|^{p}_{\omega_2,p}}{\|u\|^{p}_{\omega_1,\alpha p }}, \quad \lambda>0,
	\end{equation}
where we recall the notations
\begin{equation}\label{L_and_M}
	L=-(n+\tau_1)\gamma+(n+\tau_3)\ \ {\rm  and}\ \ M=p+\frac{n+\tau_1}{\alpha}-(n+\tau_2).
\end{equation}
We notice that by our assumptions we have $L>0$ and $M\geq 0$.
We distinguish three cases, depending on the values of $K_2$ and $M$.

\textit{Case a:} $K_2> 0$ and $M>0$. In this case we shall optimize \eqref{eq:f_lambda} w.r.t. $\lambda>0$, which occurs whenever  
$$
-K_2L\lambda^{-L-1}\frac{\|u\|^{\alpha p\gamma}_{\omega_3,\alpha p \gamma}}{\|u\|^{\alpha p\gamma}_{\omega_1,\alpha p }}+K_3M\lambda^{M-1}\frac{\|\nabla u\|^{p}_{\omega_2,p}}{\|u\|^{p}_{\omega_1,\alpha p }}=0,
$$
i.e.,
$$\lambda=\left[\frac{K_2}{K_3}\frac{L}{M}\frac{\|u\|^{\alpha p \gamma}_{\omega_3, \alpha p \gamma}}{\|\nabla u\|^p_{\omega_2, p}\, \|u\|^{p(\alpha \gamma-1)}_{\omega_1, \alpha p}}\right]^{\frac 1{L+M}}.
$$
Replacing this value into \eqref{eq:f_lambda} we get
	\begin{equation}\label{eq:2.12}
	K_1\leq K_2^{\frac{M}{L+M}}K_3^{\frac{L}{L+M}}\frac{\|\nabla u\|^{p\frac{L}{L+M}}_{\omega_2, p}\, \|u\|^{\alpha p \gamma\frac{M}{L+M}}_{\omega_3,\alpha p\gamma}}{\|u\|^{\alpha p\gamma\frac{M}{L+M}+p\frac{L}{L+M}}_{\omega_1, \alpha p}}\left[\left(\frac{M}{L}\right)^{\frac{L}{L+M}}+\left(\frac LM\right)^{\frac{M}{L+M}}\right].
	\end{equation}
	Define 
	\[
	\theta_1=\frac{\frac{L}{L+M}}{\alpha  \gamma \frac{M}{L+M}+\frac{L}{L+M}}=\frac L{\alpha \gamma M+L}=\frac{-(n+\tau_1)\gamma+n+\tau_3}{\alpha \gamma (p-\tau_2-n)+n+\tau_3}.
	\]
Note, that by our assumptions we have that $\theta_1 \in (0, 1]$.  Hence \eqref{eq:2.12} gives the constant 
	 \begin{equation}\label{C-1-constant}
	  C_1=\left(K_1^{-1}K_2^{\frac{M}{L+M}}K_3^{\frac{L}{L+M}}\left[\left(\frac{M}{L}\right)^{\frac{L}{L+M}}+\left(\frac LM\right)^{\frac{M}{L+M}}\right]\right)^\frac{1}{p\left(\alpha  \gamma \frac{M}{L+M}+\frac{L}{L+M}\right)}>0	
	 \end{equation}
	  such that
	\begin{equation}\label{GN-1-1}
		\left(\int_E |u|^{\alpha p}\omega_1dx\right)^{\frac 1{\alpha p}}\leq C_1 \left(\int_E|\nabla u|^p\omega_2 dx\right)^{\frac {\theta_1} p}\left(\int_E|u|^{\alpha p \gamma}\omega_3 dx\right)^{\frac{1-\theta_1}{\alpha p \gamma}}.
	\end{equation}
	When we want to minimize the expression $K_1^{-1}K_3^{\frac{L}{L+M}}$ as a function of $\mu>0$,  it turns out that the optimal value for $\mu>0$ will be given precisely by  
	\begin{equation}\label{mu-choice}
		\mu:=\left(\frac{C_0^{-1+(1-\gamma)n}}{(1-\gamma)\gamma\alpha}\frac{L}{M+pL}\frac{\displaystyle\int_E G^\gamma(y)\omega_1^{\frac 1{p'}}(y)\omega_2^{\frac 1p}(y)dy}{\displaystyle\int_E G(y)|y|^{p'}\omega_1(y)dy}\right)^{1/p'}.
	\end{equation}
Now, after a straightforward computation,  the constant appearing in \eqref{C-1-optimal} is nothing but 
\[
{\sf C}_1\coloneqq \inf_{G\in \mathcal G_{\omega_1,\omega_2};\int_E G\omega_1=1} C_1 .
\]	
		\textit{ Case b:} $K_2=0$. 
	According to \eqref{eq:f_lambda},  we necessarily have that $M=0$; otherwise, by letting $\lambda\to 0$ in \eqref{eq:f_lambda}, we obtain $0<K_1\leq 0$, a contradiction. In this case  \eqref{eq:f_lambda} reduces to 
	the weighted Sobolev inequality 
		\begin{equation}\label{Sobolev-particular}
			\|u\|_{\omega_1,\alpha p }\leq \left({K_1}^{-1}K_3\right)^\frac{1}{p}\|\nabla u\|_{\omega_2,p}.
		\end{equation}
	Note that  the constant $C_1>0$ from 
	\eqref{C-1-constant} formally also reduces to $\left({K_1}^{-1}K_3\right)^\frac{1}{p}$ and $\theta_1=1$, thus \eqref{Sobolev-particular} appears as the limit case into \eqref{GN-1-1}. 
	
	\textit{Case c:} $K_2> 0$ and $M=0$. In this case, by letting $\lambda\to +\infty$ in \eqref{eq:f_lambda}, we obtain exactly the same inequality as \eqref{Sobolev-particular}. The rest is the same as in \textit{Case b}.  \hfill $\square$\\

%	 
%	\noindent  the constant ${\sf C}_1$ from \eqref{C-1-constant} reduces to {\color{red}$\inf_{\mu>0}\left(\frac{K_3}{K_1}\right)^\frac{1}{p}$}, which is consistent with the Sobolev constant from inequality \eqref{Sobolev-particular}.

	\noindent \textbf{Proof of Theorem \ref{th:CKN2weights-2} $(\gamma>1$).} We can reason as in the case $\gamma<1$; by  \eqref{eq:2.7} we obtain that
	\begin{equation}\label{eq:2.16}
	\begin{aligned}
	\frac {C_0^{(1-\gamma)n}}{\gamma-1}&\int_EG^\gamma(y)\omega_1^{\frac1{p'}}(y)\omega_2^{\frac 1p}(y)dy+\frac{C_0\gamma (p-1)\mu^{p'}}{p(\gamma-1)+1}\int_E G(y)|y|^{p'}\omega_1(y)dy\\&\geq\left(\frac{1}{\gamma-1}-K\right)\int_E f^{\alpha p \gamma}\omega_3dx-\frac{C_0\gamma}{\mu^p(p(\gamma-1)+1)}\int_E |\nabla f|^p\omega_2dx.
	\end{aligned}
	\end{equation}
In particular, since $1+K(1-\gamma)>0$,  by \eqref{eq:2.16} it turns out that $0<\displaystyle\int_E f^{\alpha p \gamma}\omega_3dx<+\infty.$

	Let
	\[
\tilde K_1\coloneqq	\frac {C_0^{(1-\gamma)n}}{\gamma-1}\int_E G^\gamma(y)\omega_1(y)^{\frac1{p'}}\omega_2(y)^{\frac 1p}dy+\frac{C_0\gamma(p-1)\mu^{p'}}{p(\gamma-1)+1}\int_E G(y)|y|^{p'}\omega_1(y)dy>0,
	\]
%	whenever $G\in \mathcal G_0$, $ G\geq0$ and $\displaystyle \int_EG\omega_1dx=1$.
	 Then \eqref{eq:2.16} implies
	\begin{equation}\label{eq:beforeoptimumconverse}
	\tilde K_1\geq \tilde K_2\int_E f^{\alpha p \gamma}\omega_3 dx-\tilde K_3\int_E|\nabla f|^p\omega_2dx,
	\end{equation}
	where $$\tilde K_2=\frac 1{\gamma-1}-K\ \ {\rm  and}\ \ \tilde K_3=\frac{C_0\gamma}{\mu^p(p(\gamma-1)+1)};$$  note that we have  $\tilde K_2>0$ by using our assumption that $1+K(1-\gamma)>0$.
By \eqref{f-u} and 
 \eqref{eq:beforeoptimumconverse}, it follows that
	\begin{equation}\label{eq:renormalizedconverse}
	\tilde K_1\geq \tilde K_2\frac{\displaystyle\int_E |u|^{\alpha p \gamma}\omega_3 dx}{\left(\displaystyle\int_E |u|^{\alpha p }\omega_1 dx\right)^{\gamma}}-\tilde K_3\frac{\displaystyle\int_E|\nabla u|^p\omega_2dx}{\left(\displaystyle\int_E |u|^{\alpha p} \omega_1 dx\right)^{\frac 1\alpha}}.
	\end{equation}
	We now replace $u$ with $u_\lambda(x)=\lambda^{\frac{n+\tau_1}{\alpha p}}u(\lambda x)$, $\lambda>0$, thus $\displaystyle\int_E |u_\lambda|^{\alpha p} \omega_1dx=\int_E |u|^{\alpha p}\omega_1 dx$ and \eqref{eq:renormalizedconverse} transforms into 
	\begin{equation}\label{eq:lambdamax}
	\tilde K_1\geq \tilde K_2\lambda^{(n+\tau_1)\gamma-(n+\tau_3)}\frac{\|u\|^{\alpha p \gamma}_{\omega_3,\alpha p\gamma}}{\|u\|^{\alpha p\gamma}_{\omega_1, \alpha p}}-\tilde K_3\lambda^{\frac{n+\tau_1}{\alpha}-(n+\tau_2)+p}\frac{\|\nabla u\|^p_{\omega_2, p}}{\|u\|^p_{\omega_1, \alpha p}}.
	\end{equation}
 As  in \eqref{L-M}, we denote by  
 \[
 L\coloneqq{-(n+\tau_1)\gamma+(n+\tau_3)}\ \ {\rm  and} \ \  M\coloneqq p+\frac{n+\tau_1}{\alpha}-(n+\tau_2).
 \]
 Recall that by the observations made before, $\tau_3=\frac{\tau_2}p+\frac{\tau_1}{p'}$ and  $\tau_1(\frac1{p'}-\gamma)+\frac{\tau_2}{p}\leq 0$. The latter fact combined with the assumption implies that  $0<-L< M$. 
		
 The right-hand side of \eqref{eq:lambdamax} has then a maximum, given by
\[
\lambda=\left(\frac {-L}{ M}\frac{\tilde K_2}{\tilde K_3}\frac{\|u\|^{\alpha p \gamma}_{\omega_3, \alpha p \gamma}}{\|u\|^{\alpha p \gamma-p}_{\omega_1, \alpha p}\|\nabla u\|^p_{\omega_2, p}}\right)^{\frac1{M+ L}}.
\]
	Replacing this value inside \eqref{eq:lambdamax} we obtain
	\[
	\tilde K_1\geq \frac{\tilde K_2^{\frac{M}{M+L}}}{\tilde K_3^{\frac{-L}{{M+L}}}}\left(\left(\frac {-L} M\right)^{\frac{-L}{M+L}}-\left(\frac {-L} M\right)^{\frac{M}{M+L}}\right)\frac{\|u\|^{\alpha p \gamma \frac{M}{M+L}}_{\omega_3, \alpha p \gamma}}{\|u\|^{\alpha p \gamma \frac{M}{M+L}+p\frac{L}{M+L}}_{\omega_1,\alpha p} \|\nabla u\|^{p\frac {-L}{M+L}}_{\omega_2,p}}.
	\]
By choosing
	\[
	\theta_2=-\frac{L}{\alpha \gamma M}=\frac{(n+\tau_1)\gamma-n-\tau_3}{(n+\tau_1)\gamma+\alpha\gamma(p-n-\tau_2)}\in (0,1)
	\] and 
	\begin{equation}\label{C-2-constant}
		C_2=\left(\tilde K_1\frac{\tilde K_2^{\frac{-M}{M+L}}}{\tilde K_3^{\frac{L}{{M+L}}}}\left(\frac {-L} M\right)^{\frac{L}{M+L}}\frac M {M+L}\right)^{\frac{M+L}{\alpha p \gamma M}}>0,
	\end{equation}
we obtain that	
\begin{equation}\label{G-N-second}
		\left(\int_E |u|^{\alpha p\gamma}\omega_3 dx \right)^{\frac 1{\alpha p \gamma}}\leq {C}_2\left(\int_E |\nabla u|^p\omega_2dx\right)^{\frac{\theta_2}{p}} \left(\int_E |u|^{\alpha p}\omega_1 dx\right)^{\frac{1-\theta_2}{\alpha p}}.
\end{equation}
	
		Once we minimize the expression $\tilde K_1\tilde K_3^{\frac{-L}{L+M}}$ as a function of $\mu>0$,   the optimal value for $\mu>0$ is given  by  
	\begin{equation}\label{mu-choice-2}
		\mu:=\left(-\frac{C_0^{-1+(1-\gamma)n}}{(\gamma-1)\gamma\alpha}\frac{L}{M+pL}\frac{\displaystyle\int_E G^\gamma(y)\omega_1^{\frac 1{p'}}(y)\omega_2^{\frac 1p}(y)dy}{\displaystyle\int_E G(y)|y|^{p'}\omega_1(y)dy}\right)^{1/p'}.
	\end{equation}
	An elementary computation implies that   the constant appearing in \eqref{C-2-optimal} will be 
	$${\sf C}_2:=\inf_{G\in \mathcal G_{\omega_1,\omega_2};\int_E G\omega_1=1} C_2,$$
	which concludes the proof. 	
\hfill $\square$

\begin{remark}\label{remark-assumption}\rm (i)
	Assumption $1+K(1-\gamma)>0$ is crucial in Theorem \ref{th:CKN2weights-2}; indeed, if we  assume the contrary, then by $\gamma>1$ we obtain that $\frac{1}{1-\gamma}+K\geq 0$. Thus, the latter relation implies that inequality \eqref{eq:conditionC} looses its meaning, reducing to the trivial fact that the LHS is negative  while the RHS is positive in  condition (C). 
	
	(ii) One can observe from the proof of Theorems \ref{th:CKN2weights} and \ref{th:CKN2weights-2} that condition (C) is sufficient to be valid for a.e.\ $x\in E$ and for every $y\in E$. Moreover, the regularity of the weights can be also relaxed to be only differentiable a.e.\ on $E$; observe that for $\omega_3$ we need no such  regularity assumption. 
	
%	(iii) Let us observe that the constants ${\sf C}_1$ and ${\sf C}_{2}$ depend also on the value of the parameter $\mu$ chosen in the proof. In order to obtain the optimal  value of these constants we need to minimise these expressions as functions of $\mu$. An equivalent way to do this, is to choose the values of $\mu$ such that we obtain equality in the corresponding Young inequality involving these terms. 
\end{remark}

\begin{remark}\rm  Note that  Theorems \ref{th:CKN2weights} and \ref{th:CKN2weights-2} can be stated by replacing the Euclidean norm $|\cdot|$ on $E$ by a general norm $\|\cdot\|$, similar to \cite{CENV} and \cite{Lam}. To be more precise, let $\|\cdot\|_*$ be the  dual norm  of an arbitrary norm $\|\cdot\|$ given by $\|X\|_*=\sup_{\|Y\|\leq 1}X\cdot Y$, where $X\cdot Y$ stands for the usual Euclidean inner product of $X,Y\in \mathbb R^n$. 	In this setting,  in \eqref{eq:CKN2weights} the expression of the gradient norm is replaced by $\displaystyle \int_E \|\nabla u\|_*^p\omega_2 dx$.  In fact, the main modification of the proof is in the estimate \eqref{eq:CauchySchwartz}, where we use the inequality $X\cdot Y \leq \|X\|_* \|Y\|$ for every  $X,Y\in \mathbb R^n$, combined with the Young inequality, obtaining that
	$$
	\begin{aligned}
		& -\alpha p\int_E f^{\alpha p \gamma-1}\nabla f\cdot \nabla \phi \,\omega_2^{\frac 1p} \omega_1^{\frac 1{p'}}dx
		%	& \hphantom{\int_E F^{\gamma-1}\nabla F\cdot \nabla \phi\omega_2^{\frac 1{p}} \omega_1^{\frac 1{p'}}dx}\leq \alpha p\int_E f^{\alpha p\gamma-1}|\nabla f||\nabla \phi|\, \omega_2^{\frac 1p} \omega_1^{\frac 1{p'}}dx\\
		\leq \frac{\alpha}{\mu^p}\int_E \|\nabla f\|_*^p\omega_2dx+\frac{\alpha p \mu^{p'}}{p'}\int_E f^{(\alpha p\gamma -1)p'} \|\nabla \phi\|^{p'} \omega_1dx.
	\end{aligned}
	$$
	The rest of the proof is the same. 
%	
%	
%	\begin{eqnarray*}
%		\int_E|u(x)|^{p-1} {\nabla u(x)}\cdot\nabla \phi(x) \omega(x)dx &\leq& \int_E|u(x)|^{p-1} \|\nabla u(x)\|_*\|\nabla \phi(x)\| \omega(x)dx\\ &\leq & \left(\int_E\|\nabla u\|_*^p\omega\right)^\frac{1}{p}\left(\int_E|u|^p \|\nabla \phi\|^{p'}\omega\right)^\frac{1}{p'}.
%	\end{eqnarray*}

\end{remark}

\section{Applications}\label{section-3}

\subsection{{Sharp weighted Sobolev and Gagliardo-Nirenberg inequalities}}

\subsubsection{Weighted Sobolev inequalities with two weights}\label{subsection-PLMS} 	We shall show that the main result of \cite{BGK_PLMS} follows by applying Theorem \ref{th:CKN2weights}. To do that, let $E\subseteq \mathbb R^n$ be an open convex cone and   $\omega_1,\omega_2\colon E \to (0,+\infty)$ be two homogeneous weights with degree $\tau_1$ and $\tau_2.$ Let $q$ be (the critical exponent) defined by 
\[
\frac{\tau_1+n}{q}=\frac{\tau_2+n}{p}-1,
\]
and the fractional dimension $n_\tau>0$ given by  $\frac 1n_\tau=\frac 1p-\frac 1q$. 

Let us choose $\gamma=1-\frac{1}{n_\tau}$ in Theorem \ref{th:CKN2weights}. Since $\gamma>\max\{1-1/n,1/{p'}\}$, we have that $n_\tau>\max\{p,n\}.$ In particular, $q>0$, thus $\tau_2+n>p$. In fact, $q>p$, thus $\tau_2 \geq (1-\frac pn)\tau_1$. 
%Let us assume that , and define  Then we have that $n_\tau\geq n$. 
Choosing   $K=-\frac{1}{1-\gamma}$, we obtain that $K_2=0$ and $\theta_1=1$ in the proof of Theorem \ref{th:CKN2weights}, thus the Sobolev inequality \eqref{Sobolev-particular} holds, see also \cite{BGK_PLMS}. We notice that in this setting condition (C) reduces to \[
\left(\left(\frac{\omega_2(y)}{\omega_2(x)}\right)^{\frac 1p}\left(\frac{\omega_1(x)}{\omega_1(y)}\right)^{\frac1q}\right)^{\frac{n_\tau}{n_\tau-n}}\leq \frac{C_0}{n_\tau-n}\left(\frac1{p'} \frac{\nabla \omega_1(x)}{\omega_1(x)}+\frac{1}{p}\frac{\nabla \omega_2(x)}{\omega_2(x)}\right)\cdot y,\ \  \forall x,y\in E
\]
 whenever $n_\tau>n$, and to 
\[0\leq \left(\frac1{p'} \frac{\nabla \omega_1(x)}{\omega_1(x)}+\frac{1}{p}\frac{\nabla \omega_2(x)}{\omega_2(x)}\right)\cdot y,\ \  \forall x,y\in E
\]
in the limit case $n_\tau=n$. In fact, the above inequalities are precisely the key joint concavity conditions in \cite{BGK_PLMS}. This shows that Theorem \ref{th:CKN2weights} is a substantial extension of \cite[Theorem 1.3]{BGK_PLMS}. 

Let us note that the sharpness in the Sobolev inequality \eqref{Sobolev-particular} has been also analyzed in \cite[Theorem 1.3]{BGK_PLMS} stating that extremizers exist in \eqref{Sobolev-particular} if and only if $C_0=1$ and the weights $\omega_1$ and $\omega_2$ are equal up to a multiplicative factor. In the next section we prove a similar rigidity result in the  more general setting of the Gagliardo-Nirenberg inequalities as well. 
We also remark that the application of Proposition~\ref{prop:integrationbyparts} will fill a gap in \cite{BGK_PLMS} where we characterized the equality case in two-weighted Sobolev inequalities (see Remark~\ref{rem:PLMS}).

%\begin{definition}[\cite{BGK_PLMS}]\label{def:C-0}
%	Let $\omega_1,\omega_2\colon E \to (0,+\infty)$ be two homogeneous weights with degree $\tau_1$ and $\tau_2$ and let $q$ be defined by 
%	\[
%	\frac{\tau_1+n}{q}=\frac{\tau_2+n}{p}-1,
%	\]
%	If $\tau_2>(1-\frac pn)\tau_1$ and $n_a=\frac 1q-\frac 1p$, we say that $\omega_1$ and $\omega_2$ satisfy condition (C-0) if there exists $C_0>0$ such that
%	\[
%	\left(\left(\frac{\omega_2(y)}{\omega_2(x)}\right)^{\frac 1p}\left(\frac{\omega_1(x)}{\omega_1(y)}\right)^{\frac1q}\right)^{\frac{n_a}{n_a-n}}\leq C_0\left(\frac1{p'} \frac{\nabla \omega_1(x)}{\omega_1(x)}+\frac{1}{p}\frac{\nabla \omega_2(x)}{\omega_2(x)}\right)\cdot y,
%	\]
%	for almost every $x\in E$ and every $y\in E$.
%\end{definition}

\subsubsection{Weighted Gagliardo-Nirenberg  inequalities with one weight} 

In the sequel, we show that our main result implies sharp Gagliardo-Nirenberg inequalities with one weight $\omega\colon E\to (0,+\infty)$, covering in particular the result of Lam \cite[Theorem 1.7]{Lam}. 
 
 We recall that a  function $\psi\colon E\to \mathbb R$ is \textit{$p$-concave} on $E$, $p\in \overline{\mathbb R}$, if 
 	\[
 	\psi((1-s)x+sy)\geq \mathcal M_s^p(\psi(x),\psi(y)),\ \ \forall x,y\in E,\ s\in [0,1],
 	\]
 	where $\mathcal M_s^p$ denotes the usual $p$-mean:
 	\[
 	\mathcal M_s^p(a,b)=\begin{cases}
 		((1-s)a^p+sb^p)^{\frac 1p}& \text{ if $ab\neq 0$}\\
 		0& \text{ if $ab=0$},
 	\end{cases}
 	\]
 	with the conventions $\mathcal M^{-\infty}_s(a,b)=\min\{a,b\}$; $\mathcal M_s^0(a,b)=a^{1-s}b^s$; and $\mathcal M_s^{+\infty}(a,b)=\max\{a,b\}$ if $ab\neq 0$ and $\mathcal M_s^{+\infty}(a,b)=0$ if $ab=0$.

 	\begin{remark} \label{p-concave-monotonicity}\rm
 		(i) The $p$-concavity of $\psi$ means that $\psi^p$ is concave on $E$ if $p>0$, $\psi^p$ is convex on $E$ if $p<0$, $\psi$ is log-concave on $E$ if $p=0$, and $\psi$ is constant on $E$ if $p=+\infty$.
 		
 		(ii) Due to the monotonicity property of the $p$-mean function, the $p_1$-concavity implies the $p_2$-concavity of a function $\omega\colon E\to \mathbb R$  whenever $p_1\geq p_2$.
 	\end{remark}

 For further reference, we denote by 
 $B(\cdot,\cdot)$  the Euler-Beta function and we shall use the notation $\omega_{SE}\coloneqq\displaystyle\int_{\mathbb S^{n-1} \cap E }\omega d\mathcal H^{n-1}$ for the integral of $\omega$ over the sphere $\mathbb S^{n-1}$ intersected by the cone $E$. 

\begin{theorem}\label{th:Lam} Let $E\subseteq \mathbb R^n$ be an open convex cone,  $\omega\colon E\to (0,+\infty)$ be a homogeneous weight of class $\mathcal C^1$ with degree $\tau\geq 0$, and  $1\neq  \gamma \geq 1-\frac{1}{n+\tau}$ be such that $\omega$ is  $\frac{1-\gamma}{1-n(1-\gamma)}$-concave on $E$. Let $1<p<n+\tau$ and $\alpha\coloneqq\frac 1{p(\gamma-1)+1}$. Then the following statements hold.
	\begin{itemize}
		\item[(i)] If $\gamma<1$, then   
		\begin{equation}\label{eq:Lam1}
			\left(\int_E |u|^{\alpha p}\omega dx\right)^{\frac{1}{\alpha p}}\leq \tilde {\sf C}_1 \left(\int_E|\nabla u|^p\omega dx\right)^{\frac{\theta_1}{p}}\left(\int_E|u|^{\alpha p \gamma}\omega dx\right)^{\frac{1-\theta_1}{\alpha p \gamma}},\ \ \forall u\in \dot W^{p,\alpha}(\omega, \omega; E),
		\end{equation} 
		 where  $\theta_1=\frac{(n+\tau)(1-\gamma)}{\alpha \gamma (p-\tau-n)+n+\tau}$ and 
		 $$\tilde {\sf C}_1=\left(\frac{\alpha-1}{p'}\right)^{\theta_1}
		 \frac{\left(\frac{p'}{n+\tau}\right)^{\frac{\theta_1}{p}+\frac{\theta_1}{n+\tau}}\left(\frac{\alpha (p-1)+1}{\alpha
		 		-1}-\frac{n+\tau}{p'}\right)^\frac{1}{\alpha p}
		 	\left(\frac{\alpha (p-1)+1}{\alpha
		 		-1}\right)^{\frac{\theta_1}{p}-\frac{1}{\alpha p}}}{\left(\displaystyle\int_{B\cap E} \omega(x)dx\, {B}\left(\frac{\alpha (p-1)+1}{\alpha
		 		-1}-\frac{n+\tau}{p'},\frac{n+\tau}{p'}\right)\right)^{\frac{\theta_1}{n+\tau}}}.$$
		Moreover, equality holds in \eqref{eq:Lam1} if and only if $u$ is of the form
		\[
		w_\lambda(x)=A(\lambda+|x+x_0|^{p'})^\frac{1}{1-\alpha}, x\in E,
		\]
		where $A,\lambda >0$ and $x_0\in -\overline E\cap \overline E$ with $\omega(x+x_0)=\omega(x)$, $x\in E$.
		\item[(ii)] 
		If $\gamma>1$,  then   %(n+\tau_1)(\gamma-p\gamma+p-2)+(n+\tau_2)$,
		\begin{equation}\label{eq:Lam2}
			\left(\int_E |u|^{\alpha p\gamma}\omega dx\right)^{\frac{1}{\alpha p\gamma}}\leq \tilde  {\sf C}_2 \left(\int_E|\nabla u|^p\omega dx\right)^{\frac{\theta_2}{p}}\left(\int_E|u|^{\alpha p }\omega dx\right)^{\frac{1-\theta_2}{\alpha p }},\ \ \forall u\in \dot W^{p,\alpha}(\omega, \omega; E),
		\end{equation} 
		 where  $\theta_2=\frac{(n+\tau)(\gamma-1)}{\alpha \gamma (p-\tau-n)+\gamma(n+\tau)}$ and
			\begin{eqnarray*}
			\tilde {\sf C}_2&=& \left(\frac{1-\alpha}{p'}\right)^{\theta_2}
				\frac{\left(\frac{p'}{n+\tau}\right)^{\frac{\theta_2}{p}+\frac{\theta_2}{n+\tau}}\left(\frac{\alpha (p-1)+1}{1-\alpha
					}+\frac{n+\tau}{p'}\right)^{\frac{\theta_2}{ p}-\frac{1}{\alpha(p-1)+1}}
					\left(\frac{\alpha (p-1)+1}{1-\alpha
					}\right)^{\frac{1}{\alpha(p-1)+1}}}{\left(\displaystyle\int_{B\cap E} \omega(x)dx\, {B}\left(\frac{\alpha (p-1)+1}{1-\alpha
					},\frac{n+\tau}{p'}\right)\right)^{\frac{\theta_2}{n+\tau}}}.
		\end{eqnarray*}
	Moreover, equality holds in \eqref{eq:Lam2} for the family of functions \[
		w_\lambda(x)=A(\lambda-(1-\alpha)|x+x_0|^{p'})_+^\frac{1}{1-\alpha},\ x\in E,\ \lambda>0.
		\]
		where $A,\lambda >0$ and $x_0\in -\overline E\cap \overline E$ with $\omega(x+x_0)=\omega(x)$, $x\in E$.
	\end{itemize}
\end{theorem}

\begin{proof} (i) In Theorem \ref{th:CKN2weights}  we are going to  choose 
	$\omega_1=\omega_2=\omega_3=\omega$, $\tau_1=\tau_2=\tau_3=\tau\geq 0$,  $K=-n-\tau$ and $C_0=1.$ In the sequel we may assume that $\tau>0$; otherwise, the weights should be constants, and the results reduce to \cite{CENV}.  We also have that $\gamma \geq 1-\frac{1}{n+\tau}\geq \max\{ 1-\frac{1}{n},\frac1{p'}\}.$  With these choices, it turns out by Proposition \ref{Lam-equivalence} (see in particular inequality \eqref{Lam-3}) that the  $\frac{1-\gamma}{1-n(1-\gamma)}$-concavity of $\omega$ on $E$ implies \eqref{eq:conditionC}. In particular, by Theorem \ref{th:CKN2weights} one has the validity of the Gagliardo-Nirenberg inequality \eqref{eq:CKN2weights}, with the constant ${\sf C}_1>0$ from \eqref{C-1-optimal}. Since $M=p(1+(\gamma-1)(n+\tau))$ and $L=(1-\gamma)(n+\tau)$, a simple computation 
	shows that the constant from \eqref{C-1-1-constant-initial}
	becomes 
	$$C_{K,L,M,C_0}=\left(\frac{p\gamma\alpha}{n+\tau}\right)^L.$$
	Consequently, the constant ${\sf C}_1>0$ from \eqref{C-1-optimal} is
	\begin{equation}\label{C-representation}
		{\sf C}_1=\left(\frac{p\gamma\alpha}{n+\tau}\right)^{\theta_1}\inf_{G\in \mathcal G_{\omega,\omega};\int_E G\omega=1}\frac{\displaystyle\left(\int_E G(y)|y|^{p'}\omega(y)dy\right)^\frac{\theta_1}{p'}}{\displaystyle\left(\int_E G^\gamma(y)\omega(y)dy\right)^{\frac{\theta_1}{L}}}.
	\end{equation}
	
	We claim that ${\sf C}_1=\tilde {\sf C}_1$ whose proof is based on a ``duality'' principle. To see this, we recall that  if $q>\frac{s+n+\tau}{p'}>0$, then for every $\lambda,c>0$ we have	\begin{equation}\label{Beta-function-0}
		\int_{E}(\lambda+c|y|^{p'})^{-q}|y|^s\omega(y)dy=\lambda^{-q}\left(\frac{\lambda}{c}\right)^\frac{s+n+\tau}{p'}\frac{1}{p'}{\sc B}\left(q-\frac{s+n+\tau}{p'},\frac{s+n+\tau}{p'}\right)\omega_{SE}.
	\end{equation}
Let us choose $$
G_0(x)=(1 +|x|^{p'})^{\frac{\alpha p}{1-\alpha}}, \ x\in E.
$$
According to \eqref{Beta-function-0}, we have that 
$$I_1^0=\int_E G_0(x)\omega(x)dx=\frac{1}{p'}{\sc B}\left(\frac{\alpha p}{\alpha-1}-\frac{n+\tau}{p'},\frac{n+\tau}{p'}\right)\omega_{SE},$$
$$I_2^0=\int_E G_0(x)|x|^{p'}\omega(x)dx=\frac{1}{p'}{\sc B}\left(\frac{\alpha p}{\alpha-1}-\frac{p'+n+\tau}{p'},\frac{p'+n+\tau}{p'}\right)\omega_{SE},$$
$$I_3^0=\int_E G_0^\gamma(x)\omega(x)dx=\frac{1}{p'}{\sc B}\left(\frac{\alpha\gamma p}{\alpha-1}-\frac{n+\tau}{p'},\frac{n+\tau}{p'}\right)\omega_{SE},$$
and these expressions are well-defined due to the facts that 
$n>-\tau$ and $\frac{\alpha p}{\alpha-1}-\frac{n+\tau}{p'}>\frac{\alpha\gamma p}{\alpha-1}-\frac{n+\tau}{p'}=\frac{\alpha p}{\alpha-1}-\frac{p'+n+\tau}{p'}>0.$ Consequently, $G_0\in \mathcal G_{\omega,\omega}$. Therefore, by \eqref{C-representation} and an elementary computation  we have that
\begin{eqnarray*}
	{\sf C}_1&\leq& \left(\frac{p\gamma\alpha}{n+\tau}\right)^{\theta_1}\frac{\displaystyle\left(\int_E G_0(y)|y|^{p'}\omega(y)dy\right)^\frac{\theta_1}{p'}}{\displaystyle\left(\int_E G_0^\gamma(y)\omega(y)dy\right)^{\frac{\theta_1}{L}}}\left(\int_E G_0(y)\omega(y)dy\right)^{\frac{\gamma \theta_1}{L}-\frac{\theta_1}{p'}}\\&=&\left(\frac{p\gamma\alpha}{n+\tau}\right)^{\theta_1}\frac{(I_2^0)^\frac{\theta_1}{p'}}{(I_3^0)^{\frac{\theta_1}{L}}}\left(I_1^0\right)^{\frac{\gamma \theta_1}{L}-\frac{\theta_1}{p'}}\\&=&\tilde {\sf C}_1,
\end{eqnarray*}
where we used the recurrence relation $B(a+1,b)=\frac{a}{a+b}B(a,b)$ for every $a,b>0$, the identity $\frac{\gamma \theta_1}{L}-\frac{\theta_1}{p'}=\frac{1}{\alpha p}$ and relation	${\omega_{SE}}=({n+\tau})\displaystyle\int_{B\cap E}\omega.$
%and by using \eqref{Beta-function-0}, a long computation implies that ${\sf C}_1\leq \tilde {\sf C}_1$. 
 
 On the other hand,   the Gagliardo-Nirenberg inequality \eqref{eq:CKN2weights} holds, i.e., 
 \begin{equation}\label{sup-GN}
 	\sup_{u\in \dot{W}^{p,\alpha}(\omega,\omega;E)\setminus \{0\}}\frac{\left(\displaystyle\int_E |u|^{\alpha p}\omega dx\right)^{\frac{1}{\alpha p}}}  { \left(\displaystyle\int_E|\nabla u|^p\omega dx\right)^{\frac{\theta_1}{p}}\left(\displaystyle\int_E|u|^{\alpha p \gamma}\omega dx\right)^{\frac{1-\theta_1}{\alpha p \gamma}}}\leq {\sf C}_1.
 \end{equation}
Let $$u_0(x)=(1 +|x|^{p'})^{\frac{1}{1-\alpha}},\ x\in E.$$ 
Since $\nabla u_0(x)=\frac{p'}{1-\alpha}(1+|x|^{p'})^\frac{\alpha}{1-\alpha}|x|^{p'-2}x$, it turns out by \eqref{Beta-function-0} that
\[
\int_E|\nabla u_0|^p\omega dx=\left(\frac{p'}{\alpha-1}\right)^pI_2^0,
\]
while 
\[
\int_E  u_0 ^{\alpha p}\omega dx=I_1^0,\ \  {\rm and}\ \ \int_E u_0 ^{\alpha p \gamma}\omega dx=I_3^0.
\]
In particular, $u_0\in \dot{W}^{p,\alpha}(\omega,\omega;E).$ Therefore, by \eqref{sup-GN} and the latter relations, a similar computation as above shows that 
\[{\sf C}_1\geq \frac{\left(\displaystyle\int_E u_0^{\alpha p}\omega dx\right)^{\frac{1}{\alpha p}}}  { \left(\displaystyle\int_E|\nabla u_0|^p\omega dx\right)^{\frac{\theta_1}{p}}\left(\displaystyle\int_Eu_0^{\alpha p \gamma}\omega dx\right)^{\frac{1-\theta_1}{\alpha p \gamma}}}=\left(\frac{\alpha-1}{p'}\right)^{\theta_1}\frac{\left(I_1^0\right)^{\frac{1}{\alpha p}}}  { \left(I_2^0\right)^{\frac{\theta_1}{p}}\left(I_3^0\right)^{\frac{1-\theta_1}{\alpha p \gamma}}}=\tilde {\sf C}_1.
\]
Thus, the proof of  ${\sf C}_1=\tilde {\sf C}_1$ is concluded, which shows also the validity of \eqref{eq:Lam1}. We also notice that the value from \eqref{mu-choice} becomes $\mu=(p')^\frac{1}{p'}$, which is consistent with Lam \cite{Lam}. 

Let us focus to the equality case. First, as in the proof of the claim  ${\sf C}_1=\tilde {\sf C}_1$, a change of variables together with the property $\omega (x+x_0)=\omega(x)$ for every $x\in E$ and a direct computation easily shows that $w_\lambda(x)=(\lambda +|x+x_0|^{p'})^{\frac{1}{1-\alpha}}$ provides a class of extremizers in \eqref{eq:Lam1}.  

Conversely, we assume that a function 
 $u\colon E\to \mathbb R$ satisfies the equality in \eqref{eq:Lam1}; without loss of generality, we may assume that $u\geq 0.$  We first show that the transport map $\nabla \phi$ is more regular, namely that the singular part of the distributional Laplacian $\Delta_s\phi$ vanishes, where $\phi$ is the convex function on $E$ by the proof of Theorem \ref{th:CKN2weights}. To do so, we follow some ideas from \cite{CENV}, also iterated in \cite{BDK}. 
 
  The main idea of the proof that will follow is to trace back the proof of Theorem~\ref{th:CKN2weights} analysing the cases of equality in each of the inequalities that show up in the proof.  In doing so we are making the particular choice
		\[
		G(x)=(\lambda +|x|^{p'})^{\frac{\alpha p}{1-\alpha}}, x\in E,
		\]
		for an appropriate $\lambda$ such that $\displaystyle\int G\omega=1$. 
		By doing this analysis we will be able to show that $\nabla \phi$ will be reduced to the identity mapping up to a translation and we obtain that $u$ is equal to $G$ up to a translation. 
	
		To carry out the above program we start by considering the set $\Omega=\{x\in E:\phi(x)<+\infty \}$.  It is clear that  $\Omega$  contains the support of $u$ and, since $\phi$ is convex, $\Omega$ is a convex set. In particular, the equality in the Young inequality \eqref{eq:CauchySchwartz} gives 
		\begin{equation}\label{eq:sharpYoung1}
			Mu^{(\alpha p\gamma -1)p'}|\nabla \phi|^{p'}=|\nabla u|^p, \quad \text{a.e.\ on $\Omega$},
		\end{equation}
		for some $M>0$. 
		
		As a preliminary fact, we show that $u$ is continuous and strictly positive on its support and that ${\rm supp}(u)=\Omega$. 
		To do so, we define for any $k\in \mathbb N$ the truncation function  $u_k\coloneqq \max\{\frac 1k, u\}$. Using  \eqref{eq:sharpYoung1}  and recalling that $(\alpha p \gamma -1)p'=\alpha p$, we have
		\begin{equation}\label{eq:sharpYoungineq}
			Mu_k^{\alpha p} |\nabla\phi|^{p'}\geq Mu^{\alpha p} |\nabla\phi|^{p'}=|\nabla u|^p\geq |\nabla u_k|^p, \quad \text{a.e.\ on $\Omega$.}
		\end{equation} 
		Dividing both sides of \eqref{eq:optimalYoungineq} by $u_k^{\alpha p}$ (which is always nonzero) we obtain
		\begin{equation}\label{eq:inequ1-alpha}
			|\nabla (u_k^{1-\alpha})|^p\leq(\alpha-1)^pM |\nabla \phi|^{p'}, \quad \text{a.e. on $\Omega$}.
		\end{equation}
		Fix a compact set $\mathcal K\subset \Omega$ with nonempty interior such that there exists $x_0\in \mathcal K$ with $u(x_0)>0$. Since $\phi$ is convex, we have that $|\nabla \phi|$ is bounded on $\mathcal K$. By \eqref{eq:inequ1-alpha} also $|\nabla (u_k^{1-\alpha})|$ is bounded on $\mathcal K$. Since by definition $u_k$ are uniformly converging to $u$ on $\mathcal K$ and $u_k(x_0)=u(x_0)$ for large enough $k$. Then also $u_k^{1-\alpha}$ are uniformly converging to $u^{1-\alpha}$ on $\mathcal K$ and the map $u^{1-\alpha}$ is Lipschitz on $\mathcal K$. 
		In particular, if $\mathcal K\subset  \Omega$ is a compact set, there exists $c_{\mathcal K}>0$ such that
		\[
		u^{1-\alpha}(x)\leq c_{\mathcal K}, \quad \forall x\in \mathcal K,
		\]
		Since $\gamma<1$ implies that $\alpha>1$, and so we have
		\[
		u(x)\geq c_{\mathcal K}^{\frac{1}{1-\alpha}}, \quad \forall x\in \mathcal K.
		\]
		By the proof of Proposition~\ref{prop:integrationbyparts} (precisely \eqref{eq:GaussGreen}) we can find a sequence $(u_h)_{h\in \mathbb N}$ in $C_0^\infty(\mathbb R^n)$ converging to $u$ uniformly on $\mathcal K$ and in particular for every sufficiently large $h$ one has $u_h\geq c_{\mathcal K}^{\frac 1{1-\alpha}}/2$. Then the equality in \eqref{eq:stimaRHS-2} is achieved in the limit, and it  gives
		\[
		0=\lim_h\langle u_h^{\alpha p \gamma}\omega, \Delta_s\phi\rangle_{\mathcal D'}\geq (\min_{\mathcal K}\omega) \left(\frac{c_{\mathcal K}^{\frac 1{1-\alpha}}}{2}\right)^{\alpha p \gamma} \Delta_s\phi[\mathcal K].
		\]
		By taking an invading sequence of compact sets, we get that $\Delta_s\phi$ vanishes on $\Omega$. For a similar argument, with more details we refer the reader to \cite[Subsection~3.2]{BDK}.
		
		We can now consider the equality case in \eqref{eq:from21toC}. Using the information on the equality case in  Lemma~\ref{tr-det-inequality} and the fact that $\Delta_s\phi=0$ we have
		\[
		D^2_{\mathcal D'}\phi(x)=\left(\frac{\omega(\nabla\phi(x))}{\omega(x)}\right)^{\frac{1-\gamma}{1-n(1-\gamma)}}I_n, \quad \text{in distributional sense on $\Omega$}.
		\]
		By the Schwarz Lemma,  the only possibility is that $(\omega(\nabla\phi(x))/\omega(x))^{\frac{1-\gamma}{1-n(1-\gamma)}}$ is constant on $\Omega$. This means that we can find $c\in \mathbb R$, $x_0\in \mathbb R^n$ such that
		\begin{equation}\label{eq:considerations}
			\begin{aligned}
				&D^2_{\mathcal D'}\phi(x)=cI_n, && \text{in the sense of distribution on $\Omega$}\\
				& \nabla\phi(x)= cx+x_0, && \text{ on $\Omega$}\\
				&\omega(cx+x_0)=c^{\frac{1-n(1-\gamma)}{1-\gamma}}\omega(x),&& x\in \Omega.
			\end{aligned}
		\end{equation}
		Note that since $\phi$ is convex we have  $c>0$. Also, since $\nabla\phi$ is essentially bijective onto the support of $G$; and the support of $G$ is $E$, we have $c\Omega+x_0=E$. We aim at proving that  $E=\Omega$. Clearly, $\Omega \subseteq E$. By the Monge-Amp\`ere equation \eqref{Monge-Ampere} we can write
		\begin{equation}\label{eq:mongeamperesharp}
			u^{\alpha p}(x)\omega(x)=(\lambda+|cx+x_0|^{p'})^\frac{\alpha p}{1-\alpha}\omega(cx+x_0)c^n, \quad x\in \Omega.
		\end{equation}
		In particular,
		\begin{equation}\label{eq:finalfunction}
			u(x)=\begin{cases}A(\lambda+|cx+x_0|^{p'})^\frac{1}{1-\alpha}, &\text{if $x\in \Omega$},\\
				0& \text{otherwise}.
			\end{cases}
		\end{equation}
		for some constant $A>0$. If by contradiction $\Omega\neq E$, then $u$ would have jump discontinuities  on the boundary of a translated cone, which is a set of Hausdorff dimension $n-1$, in contradiction with the fact that $u$ is in the Sobolev space, see e.g.\ \cite[Theorems~4.17 \& 4.19]{EvansGariepy}. Since $\nabla\phi$ is essentially bijective, we have $cE+x_0=E$. In particular $-x_0\in \overline E$ and hence $x_0\in \overline E\cap -\overline E$.
		
		 We conclude the proof if we show that $\omega (x+x_0)=\omega(x)$ for every $x\in E$. Consider the last identity in \eqref{eq:considerations}. By evaluating it in $tx$ and exploiting the homogeneity of $\omega$ we can write
		\[
		\omega\left(cx+\frac{x_0}t\right)=c^{\frac{1-n(1-\gamma)}{1-\gamma}}\omega (x), \quad x\in E, t>0.
		\]
		Letting $t\to \infty$ in the previous equation we get $\omega (cx)=c^{\frac{1-n(1-\gamma)}{1-\gamma}}\omega(x)$ and hence
		$
		c^\tau=c^{\frac{1-n(1-\gamma)}{1-\gamma}}.
		$
		If $\tau\neq \frac{1-n(1-\gamma)}{1-\gamma}$ then necessarily $c=1$ and from \eqref{eq:considerations} we get $\omega(x+x_0)=\omega(x)$ for every $x\in E$, as required.
		If $\tau=\frac{1-n(1-\gamma)}{1-\gamma}$ then again by \eqref{eq:considerations} we can write
		\[
		\omega(cx+x_0)=c^{\frac{1-n(1-\gamma)}{1-\gamma}}\omega(x)=\omega(cx), \quad x\in E.
		\]
		Letting $x:=x/c$, $x\in E$, into the latter relation, we obtain   $\omega(x+x_0)=\omega(x)$  for every $x\in E$.
		
%		In this last case, the fact that $c\neq 1$ is not relevant for the shape of the optimal $u$, in fact one can change the values of $A$, $\lambda$ and $x_0$ accordingly:
%		\[
%		u(x)=Ac^{\frac{p'}{1-\alpha}}\left(\frac{\lambda}{c^{p'}}+(\alpha-1)\left|x+\frac{x_0}{c}\right|^{p'}\right)^{\frac{1}{1-\alpha}}.
%		\]

(ii) Note that $M=p(1+(\gamma-1)(n+\tau))$ and $L=(1-\gamma)(n+\tau)$, thus by 
\eqref{C-1-1-constant-initial-dual}
we have $$\tilde C_{K,L,M,C_0}=\left(\frac{p\gamma\alpha}{n+\tau}\right)^{-L}.$$
Accordingly, the constant ${\sf C}_2>0$ from \eqref{C-2-optimal} becomes
\begin{equation}\label{C-representation-2}
	{\sf C}_2=\left(\frac{p\gamma\alpha}{n+\tau}\right)^{\theta_2}\inf_{G\in \mathcal G_{\omega,\omega};\int_E G\omega=1}\frac{\displaystyle\left(\int_E G(y)|y|^{p'}\omega(y)dy\right)^\frac{\theta_2}{p'}}{\displaystyle\left(\int_E G^\gamma(y)\omega(y)dy\right)^{\frac{\theta_2}{L}}}.
\end{equation}
In order to prove that ${\sf C}_2=\tilde {\sf C}_2$ we proceed as in (i) by using  for
every $\lambda>0$, $q>-1$ and $s+n+\tau> 0$ the formula  
\begin{equation}\label{Beta-function}
	\int_{B_E(\lambda,\alpha)}(\lambda-(1-\alpha)|y|^{p'})_+^q|y|^s\omega(y)dy=\lambda^q\left(\frac{\lambda}{1-\alpha}\right)^\frac{s+n+\tau}{p'}\frac{1}{p'}{\sc B}\left(q+1,\frac{s+n+\tau}{p'}\right)\omega_{SE},
\end{equation}
where $B_E(\lambda,\alpha)=\left\{x\in E:|x|<\left(\frac{\lambda}{1-\alpha}\right)^\frac{1}{p'}\right\}$. Indeed, for the inequality 
${\sf C}_2\leq \tilde {\sf C}_2$ we use in \eqref{C-representation-2}  the function 
$G_0(x)=(1-|x|^{p'})_+^\frac{\alpha p}{1-\alpha}$, $x\in E,$  while for the converse, we consider in \eqref{eq:Lam2} the function $u_0(x)=(1-|x|^{p'})_+^\frac{1}{1-\alpha}$, $x\in E$. The rest of the proof is based on a computation; in particular, it turns out that for the family of functions $ w_\lambda(x)=A(\lambda-(1-\alpha)|x+x_0|^{p'})_+^\frac{1}{1-\alpha}$, $x\in E$,  one has equality in \eqref{eq:Lam2} for $\lambda>0$; here, 
 $x_0\in -\overline E\cap \overline E$ is such that $\omega(x+x_0)=\omega(x)$, $x\in E$.
\end{proof}

\begin{remark}\rm 	
	 We notice that Lam \cite{Lam} used  condition \eqref{Lam-1-0} in order to prove the statements in Theorem \ref{th:Lam} whenever $\gamma\neq 1$, which turns to be equivalent to  the $\frac{1-\gamma}{1-n(1-\gamma)}$-concavity of $\omega$ on $E$, see Proposition \ref{Lam-equivalence}.  Moreover,  Proposition \ref{Lam-equivalence-log} shows that in Theorem \ref{th:Lam} (i) (i.e.,  $\gamma<1$),   we can replace the assumption that $\omega$ is  $\frac{1-\gamma}{1-n(1-\gamma)}$-concave on $E$  by requiring that $\omega $ is log-concave.
	The limit case $\gamma\to 1$ is discussed in the next subsection. 
	
%	(ii) The value of the constant $\tilde C_{1}$ for the case $\gamma<1$ can be also explicitly computed and the argument is analogous to the case $\gamma>1$. 
\end{remark}

\subsection{Sharp weighted $p$-log-Sobolev inequality}\label{subsection-log-Soboev}

In this section, we state a sharp one-weighted log-Sobolev inequality by taking the limit $\gamma\to 1$ in Theorem \ref{th:Lam}.

Before doing this, we show that a limiting argument in condition (C) with $\gamma\to 1$ provides itself a rigid situation for the weights, i.e., they are equal to each other (up to a multiplicative constant) and log-concave.  This observation explains why weighted log-Sobolev inequalities are reasonably expected to be valid with the \textit{same} weights.

  Multiplying the inequality \eqref{eq:conditionC} by $(1-\gamma)$ assuming first that $\gamma<1$, and then that $\gamma>1$, the limiting $\gamma\to 1$ reduces  condition (C)  to 
\[
\omega_2^{\frac 1p}(y)\omega_1^{\frac1{p'}-1}(y)=\frac{\omega_3(x)}{\omega_1(x)},\ \ \ x,y\in E.
\]
Therefore, there exists $C>0$ such that
$\omega_3(x)=C\omega_1(x)$ and $\omega_2(x)=C^p\omega_1(x)$ for every $x\in E$. Inserting this relations into \eqref{eq:conditionC} and letting again $\gamma\to 1$, we obtain that 
\begin{equation}\label{ineq-log-2}
	\log\frac{\omega_1(y)}{\omega_1(x)}\leq K+n+C_0\frac{\nabla \omega_1(x)}{\omega_1(x)}\cdot y,\ \ \ x,y\in E.
\end{equation}
Let $y=\lambda x$ for $\lambda>0$ in the latter inequality; then $\tau\log \lambda \leq K+n+C_0\tau \lambda,$ $\lambda>0$. Optimizing the latter inequality in $\lambda>0$, it turns out that the best choice is $K=-\tau \log C_0-\tau-n$. Plugging this value into \eqref{ineq-log-2}, and reorganizing the terms it follows (according to Proposition \ref{Lam-equivalence-log}) that $\omega_1$ is log-concave on $E.$

 We may assume that $\tau>0$; otherwise, if $\tau=0$, the log-concavity of $\omega$ on $E$ is equivalent to the fact that $\omega$ is constant on $E$, which reduces to the well-known results, see e.g. \cite{delPinoDolbeault-2} and \cite{Gentil}. For $s+n+\tau>0$, by using 
 	the formula 
 \begin{equation}\label{gaussian-computation}
 	\int_Ee^{-\lambda{|x|^{p'}}}|x|^s \omega(x)dx=\frac{1}{p'}\omega_{SE}\lambda^{-\frac{n+\tau+s}{p'}}\Gamma\left(\frac{s+n+\tau}{p'}\right), 
 \end{equation}
 and further simple asymptotic properties of the Beta-function, a standard limiting argument in Theorem \ref{th:Lam} provides the following log-Sobolev inequality: 

\begin{theorem}\label{log-Sobolev}
	Let $E\subseteq \mathbb R^n$ be an open convex cone and $\omega\colon E\to (0,+\infty)$ be a  log-concave homogeneous weight of class $\mathcal C^1$ with degree $\tau\geq 0,$ and $p\in (1,n+\tau).$ Then for every function $u\in W^{1,p}(\omega;E)$ with $\displaystyle\int_E |u|^p\omega dx=1$ we have 
	\begin{equation}\label{sharp-log-Sobolev}
	\int_E |u|^p\log |u|^p \omega dx\leq \frac{n+\tau}{p}	\log\left(\mathcal L_{\omega,p}\displaystyle\int_E |\nabla u|^p\omega dx\right),
	\end{equation}
	where 
	$$\mathcal L_{\omega,p}=\frac{p}{n+\tau}\left(\frac{p-1}{e}\right)^{p-1}\left(\Gamma\left(\frac{n+\tau}{p'}+1\right)\int_{B\cap E}\omega\right)^{-\frac{p}{n+\tau}}.$$
	Equality holds in \eqref{sharp-log-Sobolev} if the extremal function belongs to the family of Gaussians
	\begin{equation}\label{Gaussian}
	u_{\lambda,x_0}(x)=\lambda^\frac{n+\tau}{pp'}\left(\Gamma\left(\frac{n+\tau}{p'}+1\right)\int_{B\cap E}\omega\right)^{-\frac{1}{p}}e^{-\lambda\frac{|x+x_0|^{p'}}{p}},\ x\in E, \ \lambda>0,
	\end{equation}
	with $x_0\in -\overline E\cap \overline E$ and $\omega(x+x_0)=\omega(x)$, $x\in E$.
	%	 Moreover, if a positive extremal function in \eqref{sharp-log-Sobolev} is radial, it has necessarily the form of $u_\lambda.$
\end{theorem}

As expected, by the limiting argument we loose the possibility to characterize the equality in  \eqref{sharp-log-Sobolev}; however, by using a direct approach from the OMT, inequality \eqref{sharp-log-Sobolev} has been recently stated by the authors in \cite{BDK} for every $p>1$, fully  characterizing also the equality cases.

\subsection{Sharp weighted Faber-Krahn and isoperimetric inequalities}\label{subsec-FaberKrahn}
We first prove a weighted Faber-Krahn inequality, by letting $\gamma\to \infty$ in \eqref{eq:Lam2}.

\begin{theorem}\label{th:faber-krahn}
	Let $E\subseteq \mathbb R^n$ be an open convex cone, $\omega\colon E\to (0,+\infty)$ be a log-concave, homogeneous weight of class $\mathcal C^1$ with degree $\tau\geq0$ and let $1<p<n+\tau$. Then for every $u\in C_c^{\infty}(\mathbb R^n)$, we have
	\begin{equation}\label{eq:faber-krahn}
		\int_E |u|\omega dx\leq {\sf C}_\infty\left(\int_E |\nabla u|^p\omega dx\right)^{\frac1p}\left(\int_{{\rm supp}(u)} \omega dx\right)^{\frac 1 {n+\tau}+\frac1{p'}},
	\end{equation}
where 
\[	{\sf C}_\infty=\left(\int_{B\cap E}\omega\right)^{-\frac 1{n+\tau}}(n+\tau)^{-\frac 1{p}}(p'+n+\tau)^{-\frac1{p'}}.\]
Moreover, equality holds in \eqref{eq:faber-krahn} for the class of functions $u_\lambda(x)=(\lambda-|x+x_0|^{p'})_+$, $x\in E,$ $\lambda>0$, where $x_0\in -\overline E\cap \overline E$ with $\omega(x+x_0)=\omega(x)$ for every $x\in E$. 
\end{theorem}

\begin{proof} Since $\omega$ is log-concave (i.e., $0$-concave), by  Remark \ref{p-concave-monotonicity} we know that it is also 
$
	\frac{1-\gamma}{1-n(1-\gamma)}$-concave for every $\gamma\geq 1$. 
	We can apply Theorem \ref{th:Lam} for every $\gamma\geq 1$ and we can let $\gamma\to \infty$ in \eqref{eq:Lam2}. It is a simple computation to verify that
	\[
	\alpha\to 0, \quad \alpha\gamma\to \frac 1p,\quad \vartheta_2\to 1, \quad \frac{1-\vartheta_2}{\alpha p}\to \frac{1}{n+\tau}+\frac1{p'},
	\] 
	as $\gamma\to \infty$. Adding up these information we obtain
	\[
	\begin{aligned}
		{\sf C}_\infty&=\lim_{\gamma\to \infty} \tilde {\sf C}_2
		&=\left(\int_{B\cap E}\omega\right)^{-\frac 1{n+\tau}}(n+\tau)^{-\frac 1{p}}(p'+n+\tau)^{-\frac1{p'}}.
	\end{aligned}
	\]
	
The equality in \eqref{eq:faber-krahn} for $u_\lambda(x)=(\lambda-|x+x_0|^{p'})_+$, $\lambda >0,$ can be easily verified by using formula \eqref{Beta-function} and the fact that  $B(2,x)=\frac 1 {x(x+1)}$  for every $x>0$.  
\end{proof}

A simple consequence of the  Faber-Krahn inequality from Theorem \ref{th:faber-krahn} is a new sharp isoperimetric-type inequality with weights. To state it, we recall that the space of functions with weighted bounded variation $BV(\omega;E)$ contains all functions $u \in L^1(\omega;E)$ such that 
\[\sup \left\{ \int_E u \text{div}(\omega X)  dx: X \in C^1_c(E; \mathbb R^n) , |X| \leq 1 \right\} < +\infty.\]
Note that ${W}^{1,1}(\omega;E)\subset BV(\omega;E)$. 
We can associate to any function $u\in BV(\omega;E)$ its  weighted variation measure
$\|Du\|_{\omega}$ similarly as in the usual non-weighted case. Using this notation and $\mathbbm{1}_S$ for the characteristic function of the nonempty set $S\subset \mathbb R^n$, we have:

%from  \cite{BDK} 
\begin{theorem}\label{thm-corollary-FK}
	Let $E\subseteq \mathbb R^n$ be an open convex cone, $\omega\colon E\to (0,+\infty)$ be a log-concave, homogeneous weight of class $\mathcal C^1$ with degree $\tau\geq0$. Then for every $u\in BV(\omega;E)$, we have
\begin{equation}\label{eq:faber-krahn0}
	\int_E |u|\omega dx\leq \tilde {\sf C}_\infty \| D(|u|)\|_\omega(E)\left(\int_{{\rm supp}(u)} \omega dx\right)^{\frac 1 {n+\tau}},
\end{equation}
where 
%$\| D(|u|)\|_\omega$ stands for the weighted  variation measure and 
\begin{equation}
	\label{constant-c-tilde-1}
	\tilde {\sf C}_\infty=\left(\int_{B\cap E}\omega\right)^{-\frac 1{n+\tau}}(n+\tau)^{-1}.
\end{equation}
	Moreover, equality holds in \eqref{eq:faber-krahn0} for the class of characteristic functions $u_\lambda(x)=\mathbbm{1}_{ B(-x_0,\lambda)\cap E}(x)$, $x\in E,$ $\lambda>0$, where $x_0\in -\overline E\cap \overline E$ with $\omega(x+x_0)=\omega(x)$ for every $x\in B(0,\lambda)\cap E$. 
\end{theorem}
\begin{proof}
	We take the limit $p\to 1$ in Theorem \ref{th:faber-krahn}. The equality in \eqref{eq:faber-krahn0} can be easily checked for the class of characteristic functions $u_\lambda(x)=\mathbbm{1}_{ B(-x_0,\lambda)\cap E}(x)$, $x\in E,$ $\lambda>0$ with the suitable properties on $x_0.$
\end{proof}

By taking $u=\mathbbm{1}_{S\cap E}$ $(S\subset \mathbb R^n)$ in Theorem \ref{thm-corollary-FK}, we obtain in an alternative way the sharp weighted isoperimetric inequality stated in \cite{BGK_PLMS}, \cite{Cabre-Ros-Oton-Serra} and \cite{Lam}. For the purpose, we recall that a measurable set $S\subset \mathbb R^n$ has bounded $\omega$-variation on $E$ if $\mathbbm{1}_{S}\in BV(\omega;E)$, moreover, its weighted
perimeter with respect to the convex cone $E$ is given by $P_\omega(S;E)=\|D\mathbbm{1}_{S}\|_\omega(E)$. 

\begin{corollary}\label{isoperimetric-weighted}
	Let $E\subseteq \mathbb R^n$ be an open convex cone, $\omega\colon E\to (0,+\infty)$ be a log-concave, homogeneous weight of class $\mathcal C^1$ with degree $\tau\geq0$. Then for every measurable set $S\subset \mathbb R^n$ with bounded $\omega$-variation on $E$, we have
	\begin{equation}\label{eq:faber-krahn0-0}
	 \tilde{\sf C}_\infty^{-1}	\left(\int_{S\cap E} \omega dx\right)^{1-\frac 1 {n+\tau}}\leq  P_\omega(S;E),
	\end{equation}
	where 
	%$\| D(|u|)\|_\omega$ stands for the weighted  variation measure and 
	$	\tilde {\sf C}_\infty$ is from \eqref{constant-c-tilde-1}. 
	Moreover, equality holds in \eqref{eq:faber-krahn0-0} for any ball of the form $S= B(-x_0,\lambda)$, $\lambda>0$, 
 where $\omega(x+x_0)=\omega(x)$ for every $x\in B(0,\lambda)\cap E$. 
\end{corollary}

\begin{remark}\rm Note that the limiting arguments in Theorems  \ref{th:faber-krahn} and \ref{thm-corollary-FK} as well as in  Corollary \ref{isoperimetric-weighted} prevent the characterization of the equality cases. A direct proof, similar to  \cite{BDK},  could give an affirmative answer to this question. 
\end{remark}

\section{Sharpness versus equality of the weights: proof of Theorem \ref{prop:equalitycase-intro}}\label{section-4}

 %Before proving the main result of the present section, we state the following: 

%\begin{lemma}\label{prop:samehomogeneity-0}
%	Let $E\subseteq \mathbb R^n$ be an open convex cone and let $F,G\colon E\to \mathbb R$ be two homogeneous functions of the same degree $a\neq 0$. Then, if there exists $A\in \mathbb R$ with
%	\[
%	F(\xi)\leq A+G(\xi), \quad \forall \xi \in E,
%	\]
%	then $A\geq0$ and $F\leq G$. In addition, if there exists $\xi_0\in E$ such that
%	$
%	F(\xi_0)=A+G(\xi_0),
%$
%	then $A=0$.
%\end{lemma}

In the sequel we are going to prove Theorem \ref{prop:equalitycase-intro}.  In fact, we prove a stronger statement, namely that the following facts hold: 
		\begin{itemize} 
		\item[(i)] $\omega_2=A\omega_1$ for some $A>0$ $($thus $\tau_1=\tau_2\eqqcolon \tau$$);$
		\item[(ii)]
		the optimal transport function, appearing in the proof of  Theorem \ref{th:CKN2weights},  is of the form $\nabla\phi(x)=cx+x_0$ for some $x_0 \in \overline E\cap (-\overline E)$ and $c >0$ such that $c^{\tau(1-\gamma) 
		}=(cC_0)^{1-n(1-\gamma)} $. In addition, either 
		\begin{itemize}
			\item[(ii1)] $\gamma=1-\frac{1}{n+\tau},$
			%			$
			%				\frac{\tau_2}{p}+\tau_1\left({\frac 1{p'}-\gamma}\right)= 1-n(1-\gamma)$, 
			thus $C_0=1$ and  $K=-\frac 1{1-\gamma}$   $($and  $\omega_3$ disappears$),$ or
			\item[(ii2)] 
			$\gamma\neq 1-\frac{1}{n+\tau},$
			%			$	\frac{\tau_2}{p}+\tau_1\left({\frac 1{p'}-\gamma}\right)\neq 1-n(1-\gamma)$, 
			thus  $\omega_3=	cC_0A^\frac{1}{p}\left(\frac{1}{1-\gamma}-n-\tau\right)\left(\frac{1}{1-\gamma}+K\right)^{-1}{\omega_1}$ for  $A>0$ from ${\rm (i)};$ 
		\end{itemize} 
		\item[(iii)] the weights $\omega_j$ are $\frac{1-\gamma}{1-n(1-\gamma)}$-concave on $E$, $j\in \{1,2,3\};$
		\item[(iv)]  $\nabla\omega_j(x)\cdot x_0=0$ for every $x\in E$ and $j\in \{1,2,3\};$
		%		\item[(v)] There exists $A>0$ such that $\omega_2=A\omega_1$ and $\omega_3=\textbf{???}\omega_1$. 
		
		%\item[(v)] either $\omega_j$ is constant for every $j\in \{1,2,3\}$ or  $x_0\in \partial E$.
	\end{itemize}
	
	Let us now assume that equality holds in \eqref{eq:CKN2weights} for some $u\in \dot{W}^{p,\alpha}(\omega_1,\omega_2;E)\setminus \{0\}$ and the infimum in \eqref{C-1-2-constant-initial} is achieved for some $G\in \mathcal G_{\omega_1,\omega_2}\setminus \{0\}$. 
 In particular, all the inequalities in the proof of Theorem~\ref{th:CKN2weights} must be equalities. The proof is divided into several steps. 
		
\underline{Step 1.} (Regularity of the transport map). This part is similar to the argument performed in the proof of Theorem \ref{th:Lam}. For completeness, we provide its proof. 	Consider the convex set $\Omega=\{x\in E:\phi(x)<+\infty \}$. 
	One can prove that  $u$ is locally Lipschitz and strictly positive in the interior of $\Omega$.
	By that equality in \eqref{eq:CauchySchwartz}, we have that 
		\begin{equation}\label{eq:optimalYoung}
	|\nabla u|^{p}\omega_2=Mu^{(\alpha p \gamma -1)p'}|\nabla \phi|^{p'}\omega_1,\quad \text{a.e.\ on $\Omega$},
	\end{equation}
	for some $M>0$.
	By considering for any $k\in \mathbb N$ the truncation $u_k\coloneqq \max\{\frac 1k, u\}$, relation  \eqref{eq:optimalYoung}  and $(\alpha p \gamma -1)p'=\alpha p$ imply that
	\begin{equation}\label{eq:optimalYoungineq}
	Mu_k^{\alpha p} |\nabla\phi|^{p'}\omega_1\geq Mu^{\alpha p} |\nabla\phi|^{p'}\omega_1=|\nabla u|^p\omega_2\geq |\nabla u_k|^p\omega_2, \quad \text{a.e.\ on $\Omega$}. 
	\end{equation} 
	 Consequently,  we obtain
	\begin{equation}\label{eq:u1-alpha}
	|\nabla (u_k^{1-\alpha})|^p\omega_2\leq(\alpha-1)^pM |\nabla \phi|^{p'}\omega_1, \quad \text{a.e. on $\Omega$}.
	\end{equation}
	Let us fix a compact set $\mathcal K\subset \Omega$ with the property that there exists $x_0\in \mathcal K$ with $u(x_0)>0$. We know that  $|\nabla \phi|$ is bounded on $\mathcal K$; in particular,  $|\nabla (u_k^{1-\alpha})|$ is also bounded on $\mathcal K$, which implies the fact that $u_k^{1-\alpha}$ are uniformly Lipschitz on $\mathcal K$. Since by definition $u_k$ are uniformly converging to $u$ on $\mathcal K$, and $u(x_0)>0$, then also $u_k^{1-\alpha}$ is converging to $u^{1-\alpha}$ uniformly on $\mathcal K$ and the map $u^{1-\alpha}$ is Lipschitz on $\mathcal K$. 
	Thus there exists $c_{\mathcal K}>0$ such that
	$u^{1-\alpha}(x)\leq c_{\mathcal K}$ for every $x\in \mathcal K$.
	Since  $\alpha>1$, it turns out that 
	$u(x)\geq c_{\mathcal K}^{\frac 1{1-\alpha}}$ for every $ x\in \mathcal K$.
	By repeating the same approximation argument of Proposition~\ref{prop:integrationbyparts} we can find a sequence $u_{h} \in C_c^{\infty}(E)$ that is converging to $u$ uniformly on $\mathcal K$ and
	\[
	u_h(x)\geq \frac{c_{\mathcal K}^{\frac 1{1-\alpha}}}2 \quad \forall x\in \mathcal K,
	\]
	for every sufficiently large $h$. By the equalities in \eqref{eq:stimaRHS} and \eqref{eq:stimaRHS-2} we get that
	\[
	0=\liminf_h \langle u_h\omega_2^{1/p}\omega_1^{1/{p'}}, \Delta_s\phi \rangle_{\mathcal D'} \geq \frac{c_k^{\frac1{1-\alpha}}}2\min_{\mathcal K}\omega_2^{1/p}\omega_1^{1/{p'}}\Delta_s\phi[\mathcal K].
	\]
	In particular, we obtain that $\Delta_s\phi=0$ on $\Omega$, which ends the proof of the regularity of $\phi.$  For a similar argument, with more details we refer the reader to \cite[Subsection~3.2]{BDK}.
	
	For further use, we introduce the cone generated by $\Omega$, i.e. 
	$${\sf Cone}_\Omega=\{\lambda x:x\in \Omega,\ \lambda>0\}.$$
		 
\underline{Step 2.} ($\omega_2=A\omega_1$ for some $A>0$ on ${\sf Cone}_\Omega$)
Consider the equality in \eqref{eq:from21toC}; by  the equality in Lemma \ref{tr-det-inequality} we have that for a.e. $x\in \Omega$, 
$$
C_0\left(\frac{\omega_2(x)}{\omega_1(x)}\right)^{\frac1p}D^2_{A}\phi(x)=\left(\frac{\omega_2(\nabla\phi(x))^{\frac 1p}\omega_1(\nabla\phi(x))^{\frac 1{p'}-\gamma}}{\omega_2(x)^{\frac np(1-\gamma)}\omega_1(x)^{\left(1-\frac np\right)(1-\gamma)}}\right)^{\frac1{1-n(1-\gamma)}}I_n,$$ 
i.e., $$
D^2_A\phi(x)=C_0^{-1}\left(\frac{\omega_2(\nabla\phi(x))^{\frac 1p}\omega_1(\nabla\phi(x))^{\frac 1{p'}-\gamma}}{\omega_2(x)^{\frac 1p}\omega_1(x)^{\frac 1{p'}-\gamma}}\right)^{\frac 1{1-n(1-\gamma)}}I_n.$$
Since $D^2_A\phi=D^2_{\mathcal D'}\phi$, by the Schwartz Lemma about the equality of the mixed partial derivatives,  the only possibility for the equality $D^2_{\mathcal D'}\phi(x)=f(x,\nabla\phi(x)) I_n$ to hold is that $f(x,\nabla \phi(x))=c$ for some real constant $c\in \mathbb R$. Since $\phi$ is convex we have that $c>0$. We can conclude that 
\begin{equation}\label{eq:phi-0}
\begin{aligned}
&D^2_{\mathcal D'}\phi(x)=c I_n\quad &&\text{in distributional sense on $\Omega$}\\
&\nabla\phi(x)=cx+x_0, \quad &&x\in \Omega
\end{aligned}
\end{equation}
for some $x_0\in \mathbb R^n$. 
%	, thus (i) is proved. 
Furthermore, this gives that for every $x\in \Omega$, we have
$$\frac{\omega_2(\nabla\phi(x))^{\frac 1p}\omega_1(\nabla\phi(x))^{\frac 1{p'}-\gamma}}{\omega_2(x)^{\frac 1p}\omega_1(x)^{\frac1{p'}-\gamma}}=(cC_0)^{1-n(1-\gamma)},$$
i.e.,
\begin{equation}\label{eq:cweights}
\omega_2(cx+x_0)^{\frac 1p}\omega_1(cx+x_0)^{\frac 1{p'}-\gamma}=(cC_0)^{1-n(1-\gamma)}\omega_2(x)^{\frac 1p}\omega_1(x)^{\frac 1{p'}-\gamma}.
\end{equation}
 Taking the gradient of the previous identity we get
 	\begin{equation}\label{eq:gradients}
 	\begin{split}
 	c\omega_2(cx+x_0)^{\frac1p}\omega_1(cx+x_0)^{\frac 1{p'}-\gamma}\left(\frac{1}{p}\frac{\nabla\omega_2(cx+x_0)}{\omega_2(cx+x_0)}+\left(\frac{1}{p'}-\gamma\right)\frac{\nabla\omega_1(cx+x_0)}{\omega_1(cx+x_0)}\right)=\\ (cC_0)^{1-n(1-\gamma)}\omega_2(x)^{\frac 1p}\omega_1(x)^{\frac 1{p'}-\gamma}\left(\frac{1}{p}\frac{\nabla\omega_2(x)}{\omega_2(x)}+\left(\frac{1}{p'}-\gamma\right)\frac{\nabla\omega_1(x)}{\omega_1(x)}\right), \quad x\in \Omega.
 	\end{split}
 	\end{equation}
 	Dividing both sides of \eqref{eq:gradients} by $c$, taking into account \eqref{eq:cweights} and the fact that $\omega_i^{-1}\nabla\omega_i$ is homogeneous of degree $-1$ we obtain
 	\begin{equation}\label{eq:gradientswithc}
 	\frac{1}{p}\frac{\nabla\omega_2(cx+x_0)}{\omega_2(cx+x_0)}+\left(\frac{1}{p'}-\gamma\right)\frac{\nabla\omega_1(cx+x_0)}{\omega_1(cx+x_0)}=\frac{1}{p}\frac{\nabla\omega_2(cx)}{\omega_2(cx)}+\left(\frac{1}{p'}-\gamma\right)\frac{\nabla\omega_1(cx)}{\omega_1(cx)}, \quad x\in \Omega.
 	\end{equation}
 	Moreover, taking the scalar product of \eqref{eq:gradients} with $cx+x_0$ we also infer
 	\begin{equation}\label{eq:scalarx_0}
 	\left(\frac{1}{p}\frac{\nabla\omega_2(x)}{\omega_2(x)}+\left(\frac{1}{p'}-\gamma\right)\frac{\nabla\omega_1(x)}{\omega_1(x)}\right)\cdot x_0=0, \quad x\in \Omega.
 	\end{equation}
 	
 \underline{Step 3.}  The idea is to explore the fact that inequality  holds in \eqref{eq:conditionCalt} for all $x,y \in E$ while equality happens in this inequality for $y=\nabla \phi(x)=cx+x_0$. Namely, for every $x\in E$ one has
 	\begin{equation}\label{equal-1}
 	\begin{split}
\left(\frac1{1-\gamma}-n\right)\left(\frac{\omega_2(cx+x_0)^{\frac 1p}\omega_1(cx+x_0)^{\frac 1{p'}-\gamma}}{\omega_2(x)^{\frac np(1-\gamma)}\omega_1(x)^{\left(1-\frac np\right)(1-\gamma)}}\right)^{\frac1{1-n(1-\gamma)}}=\\
 	=\left(\frac{1}{1-\gamma}+K\right)\frac{\omega_3(x)}{\omega_1(x)}+C_0\frac{\nabla(\omega_2^{1/p}\omega_1^{1/{p'}})(x)}{\omega_1(x)}\cdot (cx+x_0).
 	\end{split}
 	\end{equation}
 	To make use of the above observation we shall define for every $x\in E$ the function $r_x\colon E\to \mathbb R$ by letting
 	\begin{eqnarray*}
 		r_x(y)&=&\left(\frac{1}{1-\gamma}+K\right)\frac{\omega_3(x)}{\omega_1(x)}+C_0\frac{\nabla(\omega_2^{1/p}\omega_1^{1/{p'}})(x)}{\omega_1(x)}\cdot y\\&&-\left(\frac 1{1-\gamma}-n\right)\left(\frac{\omega_2^{\frac 1p}(y)\omega_1^{\frac1{p'}-\gamma}(y)}{\omega_2^{\frac np(1-\gamma)}(x)\omega_1^{\left(1-\frac np\right)(1-\gamma)}(x)}\right)^{\frac1{1-n(1-\gamma)}}.
 	\end{eqnarray*}
 	Due to \eqref{eq:conditionCalt}, one has that 
 	$r_x(y)\geq 0$ for every $x,y\in E$ and by relation 
 	\eqref{equal-1} we have that $r_x(cx+x_0)=0$ for every $x\in\Omega$. Accordingly, for every $x\in \Omega$, the function $r_x$ has its global minimum at $cx+x_0$, thus
 	$\nabla r_x(cx+x_0)=0.$ After a simple computation, the equation $\nabla r_x(cx+x_0)=0$ reduces to
 	\[
 	C_0\frac{\nabla(\omega_2^{1/p}\omega_1^{1/{p'}})(x)}{\omega_1(x)}=$$$$=\frac{1}{1-\gamma}\left(\frac{\omega_2^{\frac 1p}(cx+x_0)\omega_1^{\frac1{p'}-\gamma}(cx+x_0)}{\omega_2^{\frac np(1-\gamma)}(x)\omega_1^{\left(1-\frac np\right)(1-\gamma)}(x)}\right)^{\frac1{1-n(1-\gamma)}}\cdot\left(\frac{1}{p}\frac{\nabla \omega_2(cx+x_0)}{\omega_2(cx+x_0)}+\left(\frac{1}{p'}-\gamma\right)\frac{\nabla \omega_1(cx+x_0)}{\omega_1(cx+x_0)}\right),
 	\]
 	for every $x\in \Omega$. Applying \eqref{eq:cweights} for the first term on the RHS and \eqref{eq:gradientswithc} for the second term, we can conclude
 	\[
 	\frac{\nabla(\omega_2^{1/p}\omega_1^{1/{p'}})(x)}{\omega_1(x)}=\frac{c}{1-\gamma}\left(\frac{\omega_2(x)}{\omega_1(x)}\right)^{\frac1{p}}\cdot\left(\frac{1}{p}\frac{\nabla \omega_2(cx)}{\omega_2(cx)}+\left(\frac{1}{p'}-\gamma\right)\frac{\nabla \omega_1(cx)}{\omega_1(cx)}\right),\ x\in \Omega.
 	\]
 	By the $(-1)$-homogeneity of $\frac{\nabla \omega_i}{\omega_i}$ we infer 
 	\[
 	\frac{\nabla(\omega_2^{1/p}\omega_1^{1/{p'}})(x)}{\omega_1(x)}=\frac{1}{1-\gamma}\left(\frac{\omega_2(x)}{\omega_1(x)}\right)^{\frac1{p}}\cdot\left(\frac{1}{p}\frac{\nabla \omega_2(x)}{\omega_2(x)}+\left(\frac{1}{p'}-\gamma\right)\frac{\nabla \omega_1(x)}{\omega_1(x)}\right),\ x\in \Omega,
 	\]
 	which is equivalent to 
 	\[
 	\frac{1}{p}\frac{\nabla \omega_2(x)}{\omega_2(x)}+\frac{1}{p'}\frac{\nabla \omega_1(x)}{\omega_1(x)}=\frac 1{1-\gamma}\left(\frac{1}{p}\frac{\nabla \omega_2(x)}{\omega_2(x)}+\left(\frac{1}{p'}-\gamma\right)\frac{\nabla \omega_1(x)}{\omega_1(x)}\right),\ x\in \Omega.
 	\]
 	Reorganizing the above relation, and taking into account that $\gamma\neq 0$, we obtain that
 	\[
 	\frac{\nabla \omega_2(x)}{\omega_2(x)}=\frac{\nabla\omega_1(x)}{\omega_1(x)}, \ x\in \Omega.
 	\]
 	This implies, that there exists $A>0$ such that 
 	\begin{equation}\label{omegas=C}
 	\omega_2(x)=A\omega_1(x),\ x\in {\sf Cone}_\Omega.
 	\end{equation}
In addition to this, we conclude by \eqref{eq:scalarx_0} that
\begin{equation}\label{grad=0}
\nabla \omega_i(x)\cdot x_0=0,\ \forall x\in {\sf Cone}_\Omega,\ i\in \{1,2\}.
\end{equation}

\underline{Step 4.} (Concavity of the weights on ${\sf Cone}_\Omega$) By the previous step we have that $\tau_1 = \tau_2$. Denoting by  $\tau\coloneqq\tau_1=\tau_2$, we can combine \eqref{eq:scalarx_0}, \eqref{equal-1}, \eqref{eq:cweights} and \eqref{grad=0} to obtain that
\[
cC_0\left(\frac{1}{1-\gamma}-n-\tau\right)A^\frac{1}{p}=\left(\frac{1}{1-\gamma}+K\right)\frac{\omega_3(x)}{\omega_1(x)}, \ x\in {\sf Cone}_\Omega.
\]

There are two subcases. If $\gamma=1-\frac{1}{n+\tau}$, then $K=-\frac 1{1-\gamma}$ and $\omega_3$ disappears from the condition (C). 
Otherwise, if $\frac{1}{1-\gamma}-n-\tau\neq 0$ then $\frac{1}{1-\gamma}+K\neq 0$. Therefore, 
\begin{equation}\label{omega-3}
\omega_3(x)=	cC_0A^\frac{1}{p}\left(\frac{1}{1-\gamma}-n-\tau\right)\left(\frac{1}{1-\gamma}+K\right)^{-1}{\omega_1(x)}, \ x\in {\sf Cone}_\Omega.
\end{equation}
Replacing $\omega_2$ and $\omega_3$ from \eqref{omegas=C} and \eqref{omega-3} into condition (C), we obtain that 
\[
\left(\frac{1}{1-\gamma}-n\right)\left(\frac{\omega_1^{1-\gamma}(cy+x_0)}{\omega_1(x)^{1-\gamma}}\right)^\frac{1}{1-n(1-\gamma)}\leq cC_0\left(\frac{1}{1-\gamma}-\tau-n\right)+C_0\frac{ \nabla \omega_1(x)\cdot (cy+x_0)}{\omega_1 (x)}
\]
for all $x\in {\sf Cone}_\Omega$ and $y\in \Omega$. Using \eqref{eq:cweights}, \eqref{omegas=C} and \eqref{eq:scalarx_0} we get
\begin{equation}\label{omega-1-concave}
\left(\frac{1}{1-\gamma}-n\right)\left(\frac{\omega_1(y)}{\omega_1(x)}\right)^{\frac{1-\gamma}{1-n(1-\gamma)}}\leq \frac{1}{1-\gamma}-\tau-n+\frac{ \nabla \omega_1(x)\cdot y}{\omega_1 (x)},\quad \forall x\in {\sf Cone}_\Omega, y\in \Omega.
\end{equation}
We aim to extend \eqref{omega-1-concave} to every $y \in {\sf Cone}_\Omega$. To this end, pick any $\lambda>0$ and apply \eqref{omega-1-concave} to $\frac{1}{\lambda }x\in {\sf Cone}_\Omega$ and $y\in \Omega$. By the homogeneity of $\omega_1$ we get
\[
\left(\frac{1}{1-\gamma}-n\right)\left(\frac{\omega_1(\lambda y)}{\omega_1(x)}\right)^{\frac{1-\gamma}{1-n(1-\gamma)}}\leq \frac{1}{1-\gamma}-\tau-n+\frac{ \nabla \omega_1(x)\cdot (\lambda y)}{\omega_1 (x)},
\]
which, by the arbitrariness of $\lambda>0$, is precisely the $\frac{1-\gamma}{1-n(1-\gamma)}$-concavity of $\omega_1$ on the cone ${\sf Cone}_\Omega$, see Proposition \ref{Lam-equivalence}. 

In the sequel, we may assume  that $\tau=\tau_1=\tau_2>0$; otherwise, the above concavity property with the $0$-homogeneity of the weights implies that  $\omega_i$, $i\in \{1,2\}$ (and $\omega_3$ if it does not disappear) are all constant functions; in this case the proof becomes trivial.  

\underline{Step 5.} (Interplay between $c$ and $C_0$) We shall prove that 
\begin{equation}\label{c-C-0}
	c^{\tau(1-\gamma) 
	}=(cC_0)^{1-n(1-\gamma)}.
\end{equation}
If $\omega\colon E\to (0,+\infty)$ is defined as $\omega(x)=\omega_2(x)^{\frac 1p}\omega_1(x)^{\frac 1{p'}-\gamma}= A^{\frac{1}{p}}\omega_1^{1-\gamma}(x)$, $x\in E$, relation \eqref{eq:cweights} can be equivalently written into 
\begin{equation}\label{omega-reduced}
	\omega(cx+x_0)=(cC_0)^{1-n(1-\gamma)}\omega(x), \ x\in \Omega.
\end{equation}
 Clearly, by \eqref{omegas=C} the degree of homogeneity of $\omega$ is  $\tau(1-\gamma)$.  

Let us fix $x\in {\rm int}(\Omega)\neq \emptyset$. It turns out that there exists a small enough $\lambda_x>0$ such that $\lambda x\in \Omega$ for every $\lambda\in [1-\lambda_x,1+\lambda_x]$. Inserting $\lambda x\in \Omega$ instead of $x$ into \eqref{omega-reduced} with the range $\lambda\in [1-\lambda_x,1+\lambda_x]$, we obtain by the homogeneity of $\omega$ that
$$\omega(c\lambda x+x_0)=(cC_0)^{1-n(1-\gamma)}\omega(\lambda x)=\lambda ^{\tau(1-\gamma)}(cC_0)^{1-n(1-\gamma)}\omega(x)=\lambda ^{\tau(1-\gamma)}\omega(cx+x_0), \ x \in \Omega.$$
Consequently, we have 
$$\omega\left(c x+\frac{x_0}{\lambda}\right)=\omega(c x+x_0), \quad \forall \lambda\in [1-\lambda_x,1+\lambda_x].$$
Repeating the above argument, we obtain for every $k\in \mathbb N$ that
$$\omega\left(c x+\frac{x_0}{\lambda^k}\right)=\omega(c x+x_0),\quad  \forall \lambda\in [1-\lambda_x,1+\lambda_x].$$
Let us choose $\lambda\coloneqq 1+\lambda_x>1$ and take the limit $k\to \infty$ in the latter relation,  that yields
\[\omega\left(c x\right)=\omega(c x+x_0), \quad  \forall x \in \Omega.\]
Combining this relation with \eqref{omega-reduced} and the homogeneity of $\omega$, we obtain \eqref{c-C-0}. 

In addition, if $\gamma=1-\frac{1}{n+\tau}$ (thus  $K=-\frac 1{1-\gamma}$ and $\omega_3$ disappears from condition (C)), it turns out that $1-n(1-\gamma)=\tau(1-\gamma)>0$ and $C_0^{1-n(1-\gamma)}=1$, thus $C_0=1.$

\underline{Step 6.} (Reduction to the one-weighted case \& ${\sf Cone}_\Omega=E$) Since we have equality in  \eqref{eq:CKN2weights} for  $u\in \dot{W}^{p,\alpha}(\omega_1,\omega_2;E)\setminus \{0\}$ and the support of $u$ is a subset of $\Omega\subseteq {\sf Cone}_\Omega$, due to relations \eqref{omegas=C} and \eqref{omega-3} we have that
\begin{eqnarray}\label{equality-final}
	\nonumber	\left(\int_{{\sf Cone}_\Omega} |u|^{\alpha p}\omega_1dx\right)^{\frac{1}{\alpha p}}&=& {\sf C}_1 \left(\int_{{\sf Cone}_\Omega}|\nabla u|^p\omega_2 dx\right)^{\frac{\theta_1}{p}}\left(\int_{{\sf Cone}_\Omega}|u|^{\alpha p \gamma}\omega_3 dx\right)^{\frac{1-\theta_1}{\alpha p \gamma}}\\&=&
		{\sf C}_1\, C_{A,K,C_0} \left(\int_{{\sf Cone}_\Omega}|\nabla u|^p\omega_1 dx\right)^{\frac{\theta_1}{p}}\left(\int_{{\sf Cone}_\Omega}|u|^{\alpha p \gamma}\omega_1 dx\right)^{\frac{1-\theta_1}{\alpha p \gamma}},
\end{eqnarray}
where 
$$C_{A,K,C_0}=A^\frac{\theta_1}{p}\left(cC_0A^\frac{1}{p}\left(\frac{1}{1-\gamma}-n-\tau\right)\left(\frac{1}{1-\gamma}+K\right)^{-1}\right)^\frac{1-\theta_1}{\alpha p\gamma},$$ while
${\sf C}_1=\left(C_{K,L,M,C_0}C_{\mathcal G_{\omega_1,\omega_2}}\right)^\frac{1}{\alpha\gamma M+L}$ and $ \theta_1=\frac{L}{\alpha\gamma M+L}$ are from 
\eqref{C-1-optimal}. Note that   $M=p(1+(\gamma-1)(n+\tau))$ and $L=(1-\gamma)(n+\tau)$. Furthermore, we have that $$C_{\mathcal G_{\omega_1,A\omega_1}}=A^{-\frac{1}{p}(\frac{M}{p}+L)}C_{\mathcal G_{\omega_1,\omega_1}}=A^{-\frac{1}{p}}C_{\mathcal G_{\omega_1,\omega_1}},$$ see \eqref{C-1-2-constant-initial}, 
and 
\begin{eqnarray*}
	C_{K,L,M,C_0}&=&\left(\frac{1}{1-\gamma}+K\right)^\frac{M}{p}C_0^{L-(1-\gamma)n\left(\frac{M}{p}+L\right)}(\gamma \alpha)^L(1-\gamma)^{\frac{M}{p}+L}\frac{(M+pL)^{\frac{M}{p}+L}}{M^\frac{M}{p}L^L}\\&=&\left(\frac{1}{1-\gamma}+K\right)^\frac{M}{p}C_0^{\tau(1-\gamma)}(\gamma \alpha)^L(1-\gamma)\frac{p}{M^\frac{M}{p}L^L}, 
\end{eqnarray*}
see \eqref{C-1-1-constant-initial}. Since $-\frac{\theta_1}{L}+\theta_1+\frac{1-\theta_1}{\alpha p\gamma}=0$ (being the exponent of $A$), $\frac{M\theta_1}{pL}-\frac{1-\theta_1}{\alpha p\gamma}=0$ (being the exponent of the term $\frac{1}{1-\gamma}+K$), and $C_0^{\tau(1-\gamma)\frac{\theta_1}{L}}(cC_0)^\frac{1-\theta_1}{\alpha p\gamma}=1$ (being equivalent to \eqref{c-C-0}), after a reorganization of the terms it turns out that  
\[
{\sf C}_1\, C_{A,K,C_0}=\left(\frac{p\gamma\alpha}{n+\tau}\right)^{\theta_1}C_{\mathcal G_{\omega_1,\omega_1}}^\frac{\theta_1}{L}=\left(\frac{p\gamma\alpha}{n+\tau}\right)^{\theta_1}\inf_{G\in \mathcal G_{\omega_1,\omega_1}\setminus \{0\};\int_E G\omega_1=1}\frac{\displaystyle\left(\int_E G(y)|y|^{p'}\omega_1(y)dy\right)^\frac{\theta_1}{p'}}{\displaystyle\left(\int_E G^\gamma(y)\omega_1(y)dy\right)^{\frac{\theta_1}{L}}}.
\]
Now, we recognize the latter term that appeared  in the one-weighted case, see relation \eqref{C-representation} ($\omega_1$ being instead of $\omega$); in fact, this term is precisely the optimal constant $\tilde {\sf C}_1$ in  \eqref{eq:Lam1},  i.e., 
\[
	\left(\int_{{\sf Cone}_\Omega} |u|^{\alpha p}\omega_1dx\right)^{\frac{1}{\alpha p}}=
	\tilde {\sf C}_1 \left(\int_{{\sf Cone}_\Omega}|\nabla u|^p\omega_1 dx\right)^{\frac{\theta_1}{p}}\left(\int_{{\sf Cone}_\Omega}|u|^{\alpha p \gamma}\omega_1 dx\right)^{\frac{1-\theta_1}{\alpha p \gamma}}.
\]
By the characterization of the equality case in Theorem \ref{th:Lam}, applied on  ${\sf Cone}_\Omega$, and taking into account that the support of $u$ is in $\Omega\subseteq {\sf Cone}_\Omega\subseteq E$, we have that
\begin{equation}\label{eq:finalfunction-2}
	u(x)=\begin{cases}A(\lambda+|x+x_0|^{p'})^\frac{1}{1-\alpha}, &\text{if $x\in \Omega$},\\
		0& \text{otherwise},
	\end{cases}
\end{equation}
for some constants $A>0$ and $\lambda>0$. By \eqref{eq:finalfunction-2} it follows that 
\[
\Omega={\sf Cone}_\Omega=E;
\]
otherwise, the function $u\neq 0$ would have certain jump discontinuities on the boundary of $\Omega$, contradicting  the fact that $u$ belongs to the Sobolev space $\dot{W}^{p,\alpha}(\omega_1,\omega_2;E)=\dot{W}^{p,\alpha}(\omega_1,\omega_1;E)$,  see  e.g.\ \cite[Theorems 4.17 \& 4.19]{EvansGariepy}. Thus, the above properties  are valid in fact on $E={\sf Cone}_\Omega$.

\underline{Step 7.} (Properties of $x_0$) By \eqref{eq:cweights} and \eqref{omegas=C} we have that $\omega_i(x+x_0)=\omega_i(x)$ for every $x\in E$ and $i\in \{1,2\}$ (and also for $i=3$ whenever $\omega_3$ does not vanish). The further properties of $x_0$ follow in a similar manner as in the one-weighted setting. 
% Since $0\neq u\in $, we necessarily have that ; otherwise,  
 \hfill$\square$

\begin{remark}\label{rem:PLMS}
	{\rm We mention that the proof of Theorem~\ref{prop:equalitycase-intro} fills a gap in \cite[Theorem~1.3]{BGK_PLMS} when characterizing the equality case. In fact, the vanishing of $\Delta_s\phi$ (proven in Step 1) was missing from that argument, which focuses only to the algebraic relations involving the two weights.}
\end{remark}

\section{Final remarks and examples}\label{sec:Preliminar-0}

In this final section we discuss equivalent conditions for weights appearing in our main results and give examples satisfying our condition (C). 
 
\subsection{Characterization of $p$-concavity} In this subsection we show that the  more involved condition used by Lam \cite{Lam} presented in the introduction  is equivalent to a simpler concavity property of the weight.  
%that have been used in the literature in place of condition \eqref{eq:conditionC} and compare them. 

Let us recall that in \cite{Lam} the following condition was used.  Assume that  $E\subseteq \mathbb R^n$ is an open, convex cone, and $\omega$ a homogeneous weight $\omega\colon E\to (0,+\infty)$ of degree $\tau\geq 0$ such that  $1\neq\gamma\geq 1-\frac{1}{n+\tau}$ and satisfying the inequality 
	\begin{equation}\label{Lam-1again}
	\frac{1}{1-\gamma}\left(\frac{\omega(\nabla \varphi(x))}{\omega(x)}\right)^{1-\gamma}({\rm det}(M))^{1-\gamma}\leq \frac{1}{1-\gamma}-(n+\tau)+\frac{ \nabla \omega(x)\cdot\nabla \varphi(x)}{\omega (x)}+{\rm tr}(M),
	\end{equation}
	for  any positive definite symmetric matrix $M$ and locally Lipschitz function $\varphi$ with $\nabla \varphi(x)\in \overline E$ for any $x\in \overline E$ and $\nabla \varphi\cdot  \textbf{n}\leq 0$ on $\partial E$.

The following proposition gives equivalent formulations of the above condition. 

\begin{proposition}\label{Lam-equivalence}
	Let $E\subseteq \mathbb R^n$ be an open convex cone,   $\omega\colon E \to (0,+\infty)$ be a homogeneous weight  with degree $\tau\geq 0$ of class $\mathcal C^1$, and
	$1\neq  \gamma \geq 1-\frac{1}{n+\tau}$. Then the following  statements are equivalent: 
	\begin{itemize}
		\item[(i)] 	$\omega$ is  $\frac{1-\gamma}{1-n(1-\gamma)}$-concave on $E;$
		\item[(ii)] for every $x,y\in E$ one has
		\begin{equation}\label{Lam-3-00}
		\left(\frac{1}{1-\gamma}-n\right)\left(\frac{\omega(y)}{\omega(x)}\right)^\frac{1-\gamma}{1-n(1-\gamma)}\leq \frac{1}{1-\gamma}-(n+\tau)+\frac{ \nabla \omega(x)\cdot y}{\omega (x)};
		\end{equation}
		\item[(iii)] inequality \eqref{Lam-1again} holds. 
	\end{itemize}
	% if and only if  
	%	\begin{equation}\label{Lam-1}
	%		\frac{1}{1-\gamma}\left(\frac{\omega(\nabla \varphi(x))}{\omega(x)}\right)^{1-\gamma}({\rm det}(M))^{1-\gamma}\leq \frac{1}{1-\gamma}-(n+\tau)+\frac{ \nabla \omega(x)\cdot\nabla \varphi(x)}{\omega (x)}+{\rm tr}(M),
	%	\end{equation}
	%	for  any positive definite symmetric matrix $M$ and locally Lipschitz function $\varphi$ with $\nabla \varphi(x)\in \overline E$ for any $x\in \overline E$ and $\nabla \varphi\cdot \textbf{n}\leq 0$ on $\partial E$
	%	(in the case $\gamma=1$ the limit is considered in \eqref{Lam-1}). 
	%	Moreover, in the limit case $\gamma\to 1$, \eqref{Lam-1} holds if  and only if $\omega$ is log-concave. 
\end{proposition}

\begin{proof} 
%Due to the homogeneity of $\omega$ and Euler's theorem, the equivalence between (i) and (ii) is trivial. We are going to prove the equivalence between (ii) and (iii). 

In the proof, we distinguish two cases, according to the value of $\gamma$: 
	
	\textit{Case 1:}  $\gamma=1-\frac{1}{n}$. 	By $1-\frac{1}{n}=\gamma \geq 1-\frac{1}{n+\tau}$, we obtain that $\tau\leq 0$. By assumption, we have $\tau\geq 0$; so $\tau=0$. In particular, this case  corresponds to  $p=\frac{1-\gamma}{1-n(1-\gamma)}=+\infty$. Thus  (i) is equivalent that $\omega$ is constant. 
	
We shall prove next that in this case \eqref{Lam-1again} is also equivalent to the fact that $\omega$ is constant. 
		For any fixed $y\in E$, let $\varphi(x)=x\cdot y$, $x\in E$; then  $\nabla \varphi(x)=y\in E$. By the convexity of $E$, we clearly have that $\nabla \varphi\cdot \textbf{n}=y\cdot \textbf{n} \leq 0$ on $\partial E$. Let us also choose $M=cI_n$ with $c>0$ arbitrarily. With the above choices, inequality  \eqref{Lam-1again} reduces to 
	\[	
	n\left(\frac{\omega(y)}{\omega(x)}\right)^{1/n}c\leq \frac{ \nabla \omega(x)\cdot y}{\omega (x)}+nc.
	\] 
	Dividing by $cn>0$ and letting $c\to \infty$, we obtain that $\omega(y)\leq \omega(x)$ for every $x,y\in E$; thus $\omega=$constant. 
	Conversely, if $\omega=$constant (thus $\tau=0$) and $\gamma=1-\frac{1}{n}$, inequality \eqref{Lam-1again} trivially holds, reducing to the well known inequality $n({\rm det}(M))^{1/n}\leq {\rm tr}(M).$ 

%We can conclude this case by observing that in this situation \eqref{Lam-3-00} formally does not really make sense. However taking a limit as $\gamma-1-\frac{1}{n}\to 0$ the left side is equal to $\infty$ if $\omega(y) > \omega(x)$ which leads to a contradiction. Therefore the only possibility in this case is again that $\omega $ is constant showing that in this case all conditions are equivalent. 

	\textit{Case 2:} 
	 $1\neq \gamma>1-\frac{1}{n}$. (Clearly, we cannot have $\gamma<1-\frac{1}{n}$ since it would imply $\tau<0$, a contradiction). 
	%	We shall prove first that \eqref{Lam-1} implies  the $\frac{1-\gamma}{1-n(1-\gamma)}$-concavity of $\omega$.
	Note that if $\gamma <1$ then $\frac{1-\gamma}{1-n(1-\gamma)} >0$. Then (i) means that $\omega^{\frac{1-\gamma}{1-n(1-\gamma)}}$ is a concave function. Due to the homogeneity of $\omega$ and Euler's relation $\nabla \omega(x) \cdot x = \tau \omega(x)$ we can easily deduce that (i) and (ii) are equivalent. In the same way , if $\gamma >1$ then it follows that  $\frac{1-\gamma}{1-n(1-\gamma)} <0$ and in this case (i) says that $\omega^{\frac{1-\gamma}{1-n(1-\gamma)}}$ is a convex function. As before we can conclude the equivalence of (i) and (ii).
	
	We are now going to prove  the equivalence between (ii) and (iii).
	We assume the inequality \eqref{Lam-1again} holds. 
	As before, for any fixed $y\in E$, let $\varphi(x)=x\cdot y$, which is a convex function on $E$ and $\nabla \varphi(x)=y$. Thus,  by \eqref{Lam-1again} we have that
	\begin{equation}\label{Lam-2}
	\frac{1}{1-\gamma}\left(\frac{\omega(y)}{\omega(x)}\right)^{1-\gamma}({\rm det}(M))^{1-\gamma}\leq \frac{1}{1-\gamma}-(n+\tau)+\frac{ \nabla \omega(x)\cdot y}{\omega (x)}+{\rm tr}(M).
	\end{equation}
	Now, we fix both $x,y\in E$. 
	Let us choose the positive definite symmetric matrix $M=cI_n$ with $$c=\left(\frac{\omega(y)}{\omega(x)}\right)^\frac{1-\gamma}{1-n(1-\gamma)}>0.$$
	Then (\ref{Lam-2}) reduces to 
	\begin{equation}\label{Lam-3}
	\left(\frac{1}{1-\gamma}-n\right)\left(\frac{\omega(y)}{\omega(x)}\right)^\frac{1-\gamma}{1-n(1-\gamma)}\leq \frac{1}{1-\gamma}-(n+\tau)+\frac{ \nabla \omega(x)\cdot y}{\omega (x)},
	\end{equation}
	which is precisely \eqref{Lam-3-00}. 
	
	%IDE
	%
	%We first prove that (\ref{Lam-1}) implies the concavity of $\omega^\frac{1}{\frac{1}{1-\gamma}-n}$.

	%	
	%	\noindent	For simplicity of notation, let $t\coloneqq\frac{1}{1-\gamma}-n.$  Then (\ref{Lam-3}) can be rewritten into 
	%	\begin{equation}\label{Lam-4}
	%		t\left(\frac{\omega(y)}{\omega(x)}\right)^\frac{1}{t}\leq t-a+\frac{ \nabla \omega(x)\cdot y}{\omega (x)}.
	%	\end{equation}
	%	We claim that (\ref{Lam-4}) is equivalent to the $(1/t)$-concavity of $\omega$. First, if $t>0$ (which is equivalent to $\gamma<1$), the $(1/t)$-concavity of $\omega$ is equivalent to 
	%	the usual concavity of $f=\omega^{1/t}$, which  is characterized by
	%	\begin{equation}\label{Lam-5}
	%		f(y)-f(x)\leq \nabla f(x)\cdot (y-x),\ x,y\in E.
	%	\end{equation}
	%	Since $\nabla f(x)=\frac{1}{t}\omega(x)^{\frac{1}{t}-1}\nabla \omega(x)$ and $\nabla \omega(x)\cdot x =a \omega(x)$, a simple computation shows that  (\ref{Lam-4}) and (\ref{Lam-5}) are indeed equivalent. Second, if $t<0$ (which is equivalent to $\gamma>1$), the $(1/t)$-concavity of $\omega$ is equivalent to 
	%	the usual convexity of $g=\omega^{1/t}$, i.e., 
	%	\begin{equation}\label{Lam-6}
	%		g(y)-g(x)\geq \nabla g(x)\cdot (y-x),\ x,y\in E.
	%	\end{equation}
	%	A similar computation as above shows that 	\eqref{Lam-6} is equivalent to (\ref{Lam-4}). In conclusion, we proved that  \eqref{Lam-1} implies  the $\frac{1-\gamma}{1-n(1-\gamma)}$-concavity of $\omega$.
	
	Conversely, by keeping the above notations, let us assume that \eqref{Lam-3-00} holds and fix an arbitrary positive-definite symmetric matrix $M$. Thus, by  Lemma \ref{tr-det-inequality} and  \eqref{Lam-3-00}  it follows that
	\begin{eqnarray*}
		\frac{1}{1-\gamma}\left(\frac{\omega(y)}{\omega(x)}\right)^{1-\gamma}({\rm det}(M))^{1-\gamma}&=&\frac{1}{1-\gamma}\left[\left(\frac{\omega(y)}{\omega(x)}\right)^\frac{1-\gamma}{1-n(1-\gamma)}\right]^{1-n(1-\gamma)}({\rm det}(M))^{1-\gamma}\\&\leq &  \left(\frac{1}{1-\gamma}-n\right)\left(\frac{\omega(y)}{\omega(x)}\right)^\frac{1-\gamma}{1-n(1-\gamma)}+{{\rm tr}(M)}\\&\leq & \frac{1}{1-\gamma}-(n+\tau)+\frac{ \nabla \omega(x)\cdot y}{\omega (x)}+ {\rm tr}(M),
	\end{eqnarray*}
	%	Given any Lipschitz function $\varphi$,  in the latter relation we may choose $y:=\nabla \varphi(x)\in E$ for $x\in E$ (since Lam assumed that $\nabla \varphi(x)\in \overline E$ for any $x\in \overline E$). 
	which concludes the proof of \eqref{Lam-1again}. \end{proof}
	
	We conclude this part observing that the limit of \eqref{Lam-1again} as $\gamma\to 1$  reduces to 
	\begin{equation}\label{eq:lamgamma1}
	\log\left(\frac{\omega(\nabla\varphi(x))}{\omega(x)}\det M\right)\leq-(n+\tau)+\frac{\nabla \omega(x)\cdot \nabla\varphi(x)}{\omega(x)}+{\rm tr}(M),
	\end{equation}
	where $\varphi$ and $M$ are the same objects as above. Taking this into account, one could state an analog of Proposition \ref{Lam-equivalence} for the limit case $\gamma\to 1.$
	
	\begin{proposition}\label{Lam-equivalence-log}
		Let $E\subseteq \mathbb R^n$ be an open convex cone,   $\omega\colon E \to (0,+\infty)$ be a homogeneous weight  with degree $\tau\geq 0$ and of class $\mathcal C^1$. Then the following  statements are equivalent: 
		\begin{itemize}
			\item[(i)] 	$\omega$ is  log-concave on $E;$
			\item[(ii)] for every $x,y\in E$ one has
			\begin{equation}\label{Lam-3-11}
			\log\left(\frac{\omega(y)}{\omega(x)}\right)\leq -\tau+\frac{\nabla \omega (x)\cdot y}{\omega (x)}
			;
			\end{equation}
			\item[(iii)] inequality \eqref{eq:lamgamma1} holds;
			\item[(iv)] $\omega$ is $\frac{1}{\tau}$-concave. 
		\end{itemize}
	\end{proposition}

	\begin{proof} 	Again, the equivalence between (i) and (ii) is trivial. We now assume that inequality \eqref{eq:lamgamma1} holds. 
	%	We first notice that multiplying inequality \eqref{Lam-1} by $(1-\gamma)$ and sending $\gamma\to 1$ one always obtains $1\leq 1$ regardless that $\gamma\to1^-$ or $\gamma\to 1^+$. We then derive from \eqref{Lam-1} a first order condition. 
	%	Rearranging the terms in \eqref{Lam-1} we have
	%	\[
	%	\lim_{\gamma \to1} \left[\frac{1}{1-\gamma}\left(\frac{\omega(\nabla \varphi(x))}{\omega(x)}\right)^{1-\gamma}(\det M)^{1-\gamma}-\frac{1}{1-\gamma}\right]\leq -(n+\tau)+\frac{\nabla\omega(x)\cdot\nabla\varphi(x)}{\omega(x)}+{\rm tr}(M).
	%	\]
	%	The left-hand side of this inequality is nothing but the derivative of 
	%	\[
	%	s\mapsto\left(\frac{\omega(\nabla \phi(x))}{\omega(x)}\right)^s(\det M)^s
	%	\]
	%	at $s=0$. 
	Choosing $\varphi(x)= y\cdot x $ with $y\in E$, and $M=I_n$, inequality \eqref{eq:lamgamma1} implies precisely \eqref{Lam-3-11}. Conversely, 
	%  the latter inequality implies that for almost any $x,y\in E$ and any symmetric and positive-definite matrix $M\in \mathbb R^{n\times n}$ one has
	%	
	%	With the particular choice of  we get
	%	\[
	%	\log\left(\frac{\omega(y)}{\omega(x)}\right)\leq -a+\frac{\nabla \omega (x)\cdot y}{\omega (x)},
	%	\]
	%	and taking into account the Euler's identity $\nabla \omega (x)\cdot x=a\omega(x)$, the previous inequality is in turn equivalent to
	%	\begin{equation}\label{eq:logconc}
		%		\log\left(\frac{\omega(y)}{\omega(x)}\right)\leq \frac{\nabla \omega (x)\cdot(y-x)}{\omega(x)},
		%	\end{equation}
	%	which is precisely the $\log$-concavity of $\omega$.
	assume  that \eqref{Lam-3-11} holds. Let $M$ be any  positive definite $(n\times n)$-matrix. Using the AM-GM inequality and the fact that $\log x\leq x-1$ for every $x>0,$ it follows that 
	\[
	\log(\det M)\leq {\rm tr}(M)-n.
	\] 
	%	Indeed, 
	%	\[
	%	\log(\det M)\leq \log\left(\frac{{\rm tr}M}{n}\right)^n=n\log\left(\frac{{\rm tr}M}{n}\right)\leq n\left(\frac{{\rm tr}(M)}{n}-1\right)={\rm tr}(M)-n.
	%	\]
	Given $x,y\in E$,   by \eqref{Lam-3-11} and the latter inequality one has that
	\[
	\begin{aligned}
		\log\left(\frac{\omega(y)}{\omega(x)}\det M\right)=\log\left(\frac{\omega(y)}{\omega(x)}\right)+\log(\det M)\leq  -\tau+\frac{\nabla \omega (x)\cdot y}{\omega (x)}+{\rm tr}(M)-n,
	\end{aligned}
	\]
	which is exactly  inequality \eqref{eq:lamgamma1}.
	%
	%	Using again the Euler's identity we finally get
	%	\[
	%	\log\left(\frac{\omega(y)}{\omega(x)}\det M\right)\leq -(n+\tau)+\frac{\nabla\omega(x)\cdot y}{\omega(x)}+{\rm tr}(M),
	%	\]
	%	which is exactly \eqref{eq:lamgamma1}. 
		Finally, the equivalence between (i) and (iv) can be found in our recent work \cite{BDK}.
%		 Let us note that for general functions the implication  (iv) implies (i) trivially holds while (i) implies (iv) is generally false. The above statement shows however, that in the presence of homogeneity the implication (i) implies (iv) holds as well.
\end{proof}

\subsection{Examples of weights satisfying the main condition} In this subsection we provide some classes of weights which satisfy condition (C). 

\begin{example} {\rm (Monomials; $\gamma<1$)} \rm 
Let $p>1$. Let $\max\{1-1/n, 1/{p'}\}<\gamma<1$ and  $\tau_i,\alpha_i\geq 0$, $i=1,...,n$; $\tau=\tau_1+...+\tau_n$ and $\alpha=\alpha_1+...+\alpha_n$ be such that
	\begin{equation}\label{assumptions-2-0}
	\frac{\alpha}{p}+ {\tau}\left(\frac{1}{p'}-\gamma\right)\leq 1-n(1-\gamma)	\ \ {\rm and}\ \  \beta_i:=\frac{\alpha_i}{p}+ {\tau_i}\left(\frac{1}{p'}-\gamma\right)\geq 0,\ \ i=1,...,n.
	\end{equation}
Let 
\begin{equation}\label{monom-cond}
\delta_i:=\frac{\tau_i}{p'}+\frac{\alpha_i}{p},\ \ 	i=1,...,n. 
\end{equation}
By definition $\delta_i\geq 0$, 	with  the property that if $\delta_i= 0$ for some $i\in \{1,...,n\}$ then clearly $\tau_i=\alpha_i=0$. We consider the convex cone  
	\begin{equation}\label{E-set-0}
		E=\left\{x=(x_1,...,x_n)\in \mathbb R^n: x_i>0 \ {\rm whenever}\  \delta_i> 0\right\},
	\end{equation} and the weights $\omega_1(x)=x_1^{\tau_1}\cdots x_n^{\tau_n},  $ $\omega_2(x)=x_1^{\alpha_1}\cdots x_n^{\alpha_n}$ and $\omega_3(x)=x_1^{\delta_1}\cdots x_n^{\delta_n}$ for every  $x\in E.$ 
	One can prove that the triplet $(\omega_1,\omega_2,\omega_3)$  satisfies   condition (C) on $E$ with 
	\[C_0=\prod_{i=1}^n\left(\frac{\beta_i}{(1-\gamma)\delta_i}\right)^\frac{\beta_i}{1-n(1-\gamma)} \ \ {\rm and}\ \  K=-\frac{1}{1-\gamma}+C_0\left(-n-\tau+\frac{1}{1-\gamma}\left(1+\frac{\tau-\alpha}{p}\right)\right),
	\]
	with the convention $0^0=1$.
Indeed, by using the above choices and the generalized AM-GM inequality, for every $x=(x_1,\dots,x_n)\in E$ and $y=(y_1,\dots,y_n)\in E$ we have 
	\begin{eqnarray*}
	LHS &\coloneqq&	\left(\frac 1{1-\gamma}-n\right)\left(\frac{\omega_2^{\frac 1p}(y)\omega_1^{\frac1{p'}-\gamma}(y)}{\omega_2^{\frac np(1-\gamma)}(x)\omega_1^{\left(1-\frac np\right)(1-\gamma)}(x)}\right)^{\frac1{1-n(1-\gamma)}}\\&=&	\frac {1-n(1-\gamma)}{1-\gamma}\left(\frac{\omega_2(x)}{\omega_1(x)}\right)^\frac{1}{p}\left(\frac{\omega_2^{\frac 1p}(y)\omega_1^{\frac1{p'}-\gamma}(y)}{\omega_2^{\frac 1p}(x)\omega_1^{\frac1{p'}-\gamma}(x)}\right)^{\frac1{1-n(1-\gamma)}}\\&=&C_0
 {(1-n(1-\gamma))}\left(\frac{\omega_2(x)}{\omega_1(x)}\right)^\frac{1}{p}
 \left(\frac{1}{1-\gamma}\right)^{1-\frac{\sum_{i=1}^n\beta_i}{1-n(1-\gamma)}}\prod_{i=1}^n\left(\frac{\delta_i}{\beta_i}\frac{y_i}{x_i}\right)^{\frac{\beta_i}{1-n(1-\gamma)}}\\&\leq& 
 C_0
 {(1-n(1-\gamma))}\left(\frac{\omega_2(x)}{\omega_1(x)}\right)^\frac{1}{p}
 \left(\frac{1}{1-\gamma}\left({1-\frac{\sum_{i=1}^n\beta_i}{1-n(1-\gamma)}}\right)+{\frac{1}{1-n(1-\gamma)}}\sum_{i=1}^n {\delta_i}\frac{y_i}{x_i}\right)\\&=& 
 \left(\frac 1{1-\gamma}+K\right)\frac{\omega_3(x)}{\omega_1(x)}+C_0\frac{\nabla(\omega_2^{\frac 1p}(x)\omega_1^{\frac{1}{p'}}(x))}{\omega_1(x)}\cdot y.
	\end{eqnarray*}
	As expected, when the three weights are equal, it turns out that $C_0=1$ and $K=-n-\tau$, where $\tau$ is the common degree of homogeneity of the weights.
	
	% We also notice the consistency of the assumption \eqref{assumptions-2-0}  with the necessarily condition from \eqref{homog-cond-1}. 
		
	%With this choice of $E$, we have in fact  that $\omega(x)=x_1^{\tau_1}\cdots x_n^{\tau_n}$ and $ \sigma(x)=x_1^{\alpha_1}\cdots x_n^{\alpha_n}$ for every $x=(x_1,...,x_n)\in E$.
	%Set 
	%$$\gamma_i:=	\frac{\tau_i}{p'}+\frac{\alpha_i}{p}\ \ {\rm and}\ \  \beta_i:=\frac{\alpha_i}{p}-\frac{\tau_i}{q},\ \ i=1,...,n.$$ 
	%By (\ref{assumptions-2-0}), one has  $\gamma_i,\beta_i\geq 0,\ \ i=1,...,n$. 
\end{example}

\begin{example} {\rm ($\gamma>1$)} \rm 
	Let $E\subseteq \mathbb R^n$ be an open convex cone, and  $\omega_i\colon E\to (0,+\infty)$ be homogeneous weights  with degree $\tau_i\geq 0$ and of class $\mathcal C^1$, $i\in \{1,2,3\}$. Assume that $\gamma>\max\{1,\frac{\tau_3}{\tau_1}\}$, $\tau_3=\frac{\tau_2}{p}+\frac{\tau_1}{p'}$  and 
	\[
	\begin{aligned}
		&\inf_{x,y\in E\cap \mathbb S^{n-1}} \frac{\omega_1(x)}{\omega_3(x)}\cdot \left(\frac{\omega_2^{\frac 1p}(y)\omega_1^{\frac1{p'}-\gamma}(y)}{\omega_2^{\frac np(1-\gamma)}(x)\omega_1^{\left(1-\frac np\right)(1-\gamma)}(x)}\right)^{\frac1{1-n(1-\gamma)}}>0\\
		&\inf_{x,y\in E\cap \mathbb S^{n-1}}\frac{\nabla(\omega_2^{\frac 1p}(x)\omega_1^{\frac{1}{p'}}(x))}{\omega_3(x)}\cdot y>0.
	\end{aligned}
	\]
	Then the  triplet $(\omega_1,\omega_2,\omega_3)$  satisfies   condition {\rm (C)} on $E$ for some $K\in \mathbb R$ and $C_0>0$.
Indeed, since $\gamma>1$, rearranging \eqref{eq:conditionCalt}, we can see that it is enough to show that the function
	\[
	F(x,y)\coloneqq \left(\frac{1}{\gamma-1}+n\right)\frac{\omega_1(x)}{\omega_3(x)}\cdot \left(\frac{\omega_2^{\frac 1p}(y)\omega_1^{\frac1{p'}-\gamma}(y)}{\omega_2^{\frac np(1-\gamma)}(x)\omega_1^{\left(1-\frac np\right)(1-\gamma)}(x)}\right)^{\frac1{1-n(1-\gamma)}}+C_0\frac{\nabla(\omega_2^{\frac 1p}(x)\omega_1^{\frac{1}{p'}}(x))}{\omega_3(x)}\cdot y,
	\]
	has a positive lower bound on $E^2$. Consider $x_0,y_0\in E\cap \mathbb S^{n-1}$. Then, by assumption, we can find two constants $A, B>0$ such that
	\[
	\begin{aligned}
		F(\lambda x_0, \sigma y_0)\geq A \frac{\sigma^{[\frac{\tau_2}{p}+\tau_1(\frac1{p'}-\gamma)]\frac1{1-n(1-\gamma)}}}{\lambda^{[\tau_2\frac np(1-\gamma)+\tau_1(1-\frac{n}{p})(1-\gamma)]\frac1{1-n(1-\gamma)}+\tau_3-\tau_1}}+B\lambda^{\frac{\tau_2}p+\frac{\tau_1}{p'}-1-\tau_3}\sigma.
	\end{aligned}
	\]
	Since $\tau_3=\frac{\tau_2}{p}+\frac{\tau_1}{p'}$, a simple computation shows that 
%		\[
%	\left[\tau_2\frac np(1-\gamma)+\tau_1\left(1-\frac{n}{p}\right)(1-\gamma)\right]\frac1{1-n(1-\gamma)}+\tau_3-\tau_1=\left[\frac{\tau_1}{p}+\tau_1\left(\frac1{p'}-\gamma\right)\right]\frac1{1-n(1-\gamma)}.
%	
	\[
	F(\lambda x_0,\sigma y_0)\geq A\left(\frac\lambda\sigma\right)^{[\tau_1(\gamma-\frac1{p'})-\frac{\tau_2}p]\frac1{1-n(1-\gamma)}}+B \left(\frac\sigma\lambda\right),
	\]
	so  $F$ has a positive lower bound whenever 
	$
	\tau_1\left(\gamma-\frac1{p'}\right)-\frac{\tau_2}{p}>0,
	$
	which is equivalent to our assumption $\gamma> \frac{\tau_3}{\tau_1}$.
\end{example}

\begin{example} {\rm ($\gamma\neq 1$)} \rm 
Let $p>1$, $1\neq \gamma\geq 1-\frac{1}{n}$,  $E\subseteq \mathbb R^n$ be an open convex cone, and  $\omega_i\colon E\to (0,+\infty)$ be homogeneous weights  with degree $\tau_i\geq 0$ and of class $\mathcal C^1$, $i\in \{1,2,3\}$. Assume in addition that: 
\begin{itemize}
	\item[(i)] $\omega_2^\frac{1}{p}\omega_1^{\frac{1}{p'}-\gamma}$ is $\frac{1-\gamma}{1-n(1-\gamma)}$-concave on $E;$ 
	\item[(ii)] if $\gamma >1$ then $1>(1-\gamma)\left(n+\frac{\tau_1}{p}+ {\tau_2}\left(\frac{1}{p'}-\gamma\right)\right);$
	\item[(iii)] $\omega_3\geq \omega_2^\frac{1}{p}\omega_1^{\frac{1}{p'}}$ on $E;$
	\item[(iv)] $\nabla \omega_1(x)\cdot y\geq 0$ for every $x,y\in E.$
\end{itemize}
Using Proposition \ref{Lam-equivalence} and some simple estimates, one can prove that  the triplet $(\omega_1,\omega_2,\omega_3)$  satisfies   condition (C) on $E$ with $C_0=1$ and $K=-n-\frac{\tau_1}{p}- {\tau_2}\left(\frac{1}{p'}-\gamma\right)$. The details are left to the interested reader. 
\end{example}

\noindent {\bf Acknowledgment.} 
The authors thank the anonymous Referee for all valuable comments and suggestions that significantly improved the presentation of the paper.

\end{document}